\documentclass[reqno]{amsart}
\usepackage[left=2.6cm,right=2.6cm,top=3cm,bottom=3cm]{geometry}
\usepackage[backref=page]{hyperref}
\usepackage{cite}
\usepackage{enumitem}
\usepackage{graphicx}   
\usepackage{caption} 
\usepackage{subcaption}
\usepackage{xcolor}

\newcommand{\rev}[1]{#1}

\theoremstyle{plain}
\newtheorem{theorem}{Theorem}
\newtheorem{proposition}[theorem]{Proposition}

\newtheorem{lemma}[theorem]{Lemma}

\theoremstyle{definition}

\theoremstyle{remark}
\newtheorem{remark}{Remark}

\allowdisplaybreaks

\begin{document}
	
	\author{Phuoc Vinh Dinh}
	\address{Phuoc Vinh Dinh$^{1,2,3}$ (ORCID: 0009-0000-6297-7798)\newline
		$^1$Faculty of Applied Sciences, Ho Chi Minh City University of Technology (HCMUT), 268 Ly Thuong Kiet Street, Dien Hong Ward, Ho Chi Minh City, Vietnam;\newline
		$^2$Vietnam National University, Ho Chi Minh City, Vietnam\newline
		$^3$Department of Mathematics, FPT University, Ho Chi Minh City, Vietnam}
	\email{dpvinh.sdh251@hcmut.edu.vn, vinhdp2@fe.edu.vn}
	
	\author{Phuong Le}
	\address{Phuong Le$^{4,2}$ (ORCID: 0000-0003-4724-7118)\newline
		$^4$Faculty of Economic Mathematics, University of Economics and Law, Ho Chi Minh City, Vietnam; \newline
		$^2$Vietnam National University, Ho Chi Minh City, Vietnam}
	\email{phuongl@uel.edu.vn}
	
	\author{Tien Dung Nguyen}
	\address{Tien Dung Nguyen$^{1,2}$\newline
		$^1$Faculty of Applied Sciences, Ho Chi Minh City University of Technology (HCMUT), 268 Ly Thuong Kiet Street, Dien Hong Ward, Ho Chi Minh City, Vietnam;\newline
		$^2$Vietnam National University, Ho Chi Minh City, Vietnam}
	\email{dungnt@hcmut.edu.vn}
	
	
	\title[A free boundary problem for spreading of an invasive species]{A free boundary problem for spreading of an invasive species in the territory of a native competitor under shifting climate}

	\date{\today}

	\begin{abstract}
		In this paper we consider a free boundary problem for the spreading of an invasive species in the habitat of a native competitor. We assume that the invasive species benefits from a climate shift that turns the environment from unfavorable to favorable at a constant speed $c$. We show that if the invasive species is an inferior one, then it must vanish, while the native competitor always persists. However, if the invasive species is a superior one, then a dichotomy occurs: either the invasive species vanishes and the native one persists, or the invasive one spreads and the native one vanishes. We also show that in the latter case the asymptotic spreading speed equals $\min\{c, c_0\}$, where $c_0$ is the spreading speed in the corresponding homogeneous environment. Numerical simulations are provided to illustrate our theoretical results.
	\end{abstract}

	
	\dedicatory{}
	
	
	\maketitle
	
	\section{Introduction and statement of the results}\label{sec_1}
	
	Understanding how species respond to climate change is a central topic in ecology and mathematical biology. A classical approach uses reaction–diffusion equations to model the shifting of habitats and the spread of populations \cite{MR2471053,MR3377515,MR3691993}. When a species expands into a new territory, the boundary of its habitat is often unknown in advance, leading to the introduction of \emph{free boundary problems}. A seminal work in this direction is the paper by Du and Lin~\cite{MR2607347}, where a diffusive logistic model with a free boundary was proposed to describe the spreading of a single species. The authors established a sharp spreading–vanishing dichotomy: the species either spreads to the whole space (and stabilizes to a positive equilibrium) or vanishes as time goes to infinity, depending on the initial habitat size and the expansion capability of the species.
	
	In many realistic situations, however, the environment itself is not stationary but changes over time due to climate warming. To capture this effect, Hu et al.~\cite{HU20205931} (see also \cite{MR3871607,MR3639148}) incorporated a shifting habitat into the free boundary framework. They studied a model of the form
	\[
	\begin{cases}
		u_t = du_{xx} + (A(x - ct) - bu) u, & t > 0,~ 0 < x < h(t), \\
		u_x(t, 0) = u(t, h(t)) = 0, & t > 0, \\
		h'(t) = -\mu(A(h(t) - ct))u_x(t, h(t)), & t > 0, \\
		h(0) = h_0, \; u(0, x) = u_0(x), & 0 \leq x \leq h_0,
	\end{cases}
	\]
	where $u(t,x)$ is the population density, $h(t)$ the free boundary, and $A(x-ct)$ describes an environment that shifts at constant speed $c>0$. The function $A$ is assumed to be Lipschitz continuous on $\mathbb{R}$, strictly decreasing on $[0,l_0]$, and satisfies
	\[
	A(\xi) =
	\begin{cases}
		a_1 & \text{if } \xi \le 0, \\
		a_0 & \text{if } \xi \ge l_0,
	\end{cases}
	\]
	with $l_0 > 0$ and $a_0 < 0 < a_1$. Thus the environment changes from unfavorable ($A \le a_0<0$) to favorable ($A \ge a_1>0$) at speed $c$. The function $\mu$ is Lipschitz continuous and increasing on $[a_0,a_1]$, with $0<\mu(a_0)\le\mu(a_1)$, and it governs the resistance of the boundary to expansion. We shall keep the same structural assumptions on $A$ and $\mu$ throughout this paper. \rev{It is convenient, and clearly harmless, to extend $\mu$ to the whole of $\mathbb{R}$ by setting $\mu(s):=\mu(a_0)$ for $s<a_0$ and $\mu(s):=\mu(a_1)$ for $s>a_1$. The extended function is again Lipschitz continuous and nondecreasing. This extension will only be used in Section~\ref{sec_4}, where semi-waves associated with slightly perturbed growth rates are considered.}
	
	For a single species, Hu et al.~\cite{HU20205931} proved that the spreading–vanishing dichotomy persists and, when spreading happens, the asymptotic spreading speed is the minimum of the climate shifting speed $c$ and the intrinsic speed of the species in a homogeneous environment. Related free boundary problems have been further investigated in \cite{MR3764580,MR3691993,MR3377515,MR2471053,MR4404216,MR4873231,MR4766631}, among others.
	
	Another important extension is to consider the interaction between an invasive species and a native competitor. In the absence of climate shift, Du and Lin~\cite{MR3327894} studied a two–species competitive free boundary problem where the invasive species $u$ occupies a region $[0,h(t)]$ and the resident species $v$ lives in the whole half–line $[0,\infty)$. The system reads
	\begin{equation}\label{main2}
	\begin{cases}
		u_t = d_1u_{xx} + (a_1 - b_1u - c_1 v) u, & t > 0,~ 0 < x < h(t), \\
		v_t = d_2v_{xx} + (a_2 - b_2u - c_2v)v,   & t > 0,~ 0 < x < +\infty, \\
		u_x(t, 0) = v_x(t, 0) = 0, \; u(t, x) = 0, & t > 0, ~ h(t) \leq x < +\infty, \\
		h'(t) = -\mu(a_1)u_x(t, h(t)),            & t > 0, \\
		h(0)=h_0,~ u(0, x) = u_0(x),             & 0 \leq x \leq h_0, \\
		v(0, x) = v_0(x),                         & 0 \le x < +\infty.
	\end{cases}
	\end{equation}
	According to the competitive ability of the invader, two distinct scenarios are identified:
	\begin{itemize}
		\item \textbf{Inferior competitor:} $\displaystyle\frac{a_1}{a_2} < \min\Bigl\{\frac{b_1}{b_2},\frac{c_1}{c_2}\Bigr\}$. In this case the invader $u$ always disappears, while $v$ converges to the carrying capacity $a_2/c_2$.
		\item \textbf{Superior competitor:} $\displaystyle\frac{a_1}{a_2} > \max\Bigl\{\frac{b_1}{b_2},\frac{c_1}{c_2}\Bigr\}$. Then a spreading–vanishing dichotomy occurs: either $u$ vanishes or it spreads successfully, driving $v$ to extinction.
	\end{itemize}
	Later, the spreading speed of the superior invader was determined by Du, Wang and Zhou~\cite{MR3609207} (see also~\cite{MR3986328} for the weak competition case). They showed that when spreading happens, there exists a unique asymptotic speed $c_0>0$ such that $h(t)/t \to c_0$, and $c_0$ can be characterized via a semi–wave solution. In fact, $c_0$ depends on the parameter $\hat\mu = \frac{a_1}{b_1d_2}\mu(a_1)$ and is strictly increasing in $\hat\mu$, with $c^* := \lim_{\hat\mu\to+\infty}c_0 < +\infty$.
	
	The above studies either incorporate climate change for a single species or consider two–species competition in a static environment. The simultaneous effect of a shifting climate and inter-specific competition in a free boundary setting was first addressed by Lei et al.~\cite{Lei2018}. They considered two competing species that both expand according to a free boundary but are subject to a common shifting habitat. It was shown that the inferior competitor always goes extinct, while for superior competitors the spreading–vanishing dichotomy still holds. However, in many real invasions, the native species is already well established across the whole habitat and does not experience a shifting boundary; only the invasive species faces a climate–driven moving frontier.
	
	Motivated by this observation, in the present paper we study the problem
	\begin{equation}\label{main}
		\begin{cases}
			u_t = d_1u_{xx} + (A(x - ct) - b_1u - c_1 v) u, & t > 0,~ 0 < x < h(t), \\
			v_t = d_2v_{xx} + (a_2 - b_2u - c_2v)v,          & t > 0,~ 0 < x < +\infty, \\
			u_x(t, 0) = v_x(t, 0) = 0, ~ u(t, x) = 0,        & t > 0, ~ h(t) \leq x < +\infty, \\
			h'(t) = -\mu (A(h(t) - ct))u_x(t, h(t)),          & t > 0, \\
			h(0)=h_0,~ u(0, x) = u_0(x),                       & 0 \leq x \leq h_0, \\
			v(0, x) = v_0(x),                                  & 0 \le x < +\infty,
		\end{cases}
	\end{equation}
	where $d_1,d_2,a_2,b_1,b_2,c_1,c_2$ and $h_0$ are positive constants, and the initial functions satisfy
	\begin{equation}\label{initial}
		\begin{cases}
			u_0 \in C^2([0,h_0]),\quad u_0 > 0 \text{ in } [0,h_0), ~ u_0'(0) = u_0(h_0) = 0,\\
			v_0 \in C^2([0,+\infty)) \cap L^\infty([0,+\infty)),\quad v_0 > 0 \text{ in } [0,+\infty), ~ v_0'(0) = 0.
		\end{cases}
	\end{equation}
	Here only the invasive species $u$ feels the shifting climate $A(x-ct)$, while the native competitor $v$ lives in a globally favorable environment with constant growth rate $a_2$. This configuration is new and presents additional mathematical challenges compared to the purely shifting or purely competitive cases.
	
	When the shifting term $A$ is replaced by the constant $a_1$, system \eqref{main} reduces to \eqref{main2} (the model of Du and Lin~\cite{MR3327894}). To state our results, we keep the same notions of inferior and superior competitor:
	\begin{equation}\label{inferior}
		\frac{a_1}{a_2} < \min\left\{\frac{b_1}{b_2}, \frac{c_1}{c_2}\right\},
	\end{equation}
	\begin{equation}\label{superior}
		\frac{a_1}{a_2} > \max\left\{\frac{b_1}{b_2}, \frac{c_1}{c_2}\right\}.
	\end{equation}
	
	Our first main theorem deals with the inferior competitor.
	\begin{theorem}\label{th:inferior}
		Let $(u, v, h)$ be the unique solution of \eqref{main} with initial functions satisfying \eqref{initial}. Suppose that \eqref{inferior} holds. Then $u$ vanishes in the long run, that is,
		\[
		\lim_{t\to\infty} (u(t,\cdot), v(t,\cdot))
		= \left(0, \frac{a_2}{c_2}\right) \text{ in } L^\infty_{\rm loc}([0,+\infty)).
		\]
		Moreover, if $\inf_{x\in[0,+\infty)} v_0(x) > 0$, then $\lim_{t\to\infty} h(t) < +\infty$ and
		\[
		\lim_{t\to\infty} u(t,x) = 0,\quad \lim_{t\to\infty} v(t,x) = \frac{a_2}{c_2} \text{ uniformly in } [0, \infty).
		\]
	\end{theorem}
	
	\begin{remark}
		The weaker $L^\infty_{\rm loc}$ convergence in Theorem~\ref{th:inferior} requires no additional assumption on $v_0$; for the stronger uniform-in-$x$ convergence, one needs the invasive species $u$ to be competing against a native population that is present everywhere, which is ensured by $\inf_{x\in[0,+\infty)} v_0(x) > 0$.
	\end{remark}
	
	For the superior competitor we obtain a spreading–vanishing dichotomy and, more importantly, an explicit formula for the spreading speed that combines the climate shifting speed $c$ and the intrinsic spreading speed $c_0$ of the homogeneous system.
	
	\begin{theorem}\label{th:superior}
		Let $(u, v, h)$ be the unique solution of \eqref{main} with initial functions satisfying \eqref{initial}. Suppose that \eqref{superior} holds. Then either one of the following cases happens:
		\begin{enumerate}
			\item[(i)] \textit{Vanishing of $u$:} $\displaystyle\lim_{t\to\infty} h(t) \le \frac{\pi}{2}\sqrt{\frac{d_1}{a_1-a_2c_1/c_2}}$ and
			$\displaystyle\lim_{t\to\infty} \|u(t,\cdot)\|_{C([0,h(t)])} = 0$,
			$\displaystyle\lim_{t\to\infty} v(t,x) = \frac{a_2}{c_2}$ uniformly in any compact subset of $[0, \infty)$;
			\item[(ii)] \textit{Spreading of $u$:} $\displaystyle\lim_{t\to\infty} h(t) = +\infty$ and
			$\displaystyle\lim_{t\to\infty} (u(t,x), v(t,x)) = \left(\frac{a_1}{b_1},0\right)$ uniformly in any compact subset of $[0, \infty)$. If we further assume
			$\displaystyle\liminf_{x\to\infty} v_0(x) > 0$, then
			\[
			\lim_{t\to\infty} \frac{h(t)}{t} = \min\{c, c_0\}.
			\]
		\end{enumerate}
	\end{theorem}
	
	\begin{remark}\rev{\label{rm:speed}}
		The formula $\lim_{t\to\infty} h(t)/t = \min\{c,c_0\}$ in Theorem~\ref{th:superior}(ii) has a clear ecological interpretation. The quantity $c_0$ is the intrinsic spreading speed of $u$ in the absence of a climate shift constraint, while $c$ is the speed at which the favourable habitat itself moves. When $c_0 \le c$, the invasive front advances at its own intrinsic speed; when $c_0 > c$, the front cannot outrun the shifting habitat boundary and is limited to speed $c$. Hence the effective spreading speed is $\min\{c, c_0\}$.

		\rev{The native species slows the invader down as well. Let $c_0^{\rm sing}$
		denote the asymptotic spreading speed of the single species free boundary
		problem obtained from \eqref{main2} by deleting $v$, that is, of
		\[
		\hat u_t=d_1\hat u_{xx}+(a_1-b_1\hat u)\hat u \ \text{ in } (0,\hat h(t)),\quad
		\hat u_x(t,0)=\hat u(t,\hat h(t))=0,\quad \hat h'(t)=-\mu(a_1)\hat u_x(t,\hat h(t)),
		\]
		with the same initial data (see \cite{MR2607347,MR4404216}). Since
		$\tilde v>0$, the solution $(\tilde u,\tilde v,\tilde h)$ of \eqref{main2}
		satisfies
		$\tilde u_t\le d_1\tilde u_{xx}+(a_1-b_1\tilde u)\tilde u$ together with
		$\tilde u_x(t,0)=0$ and
		$\tilde h'(t)=-\mu(a_1)\tilde u_x(t,\tilde h(t))$, so it is a lower solution
		of the above problem; hence $\tilde h\le\hat h$ and, when spreading occurs,
		$c_0\le c_0^{\rm sing}$. The invasion is therefore retarded by two
		independent mechanisms, interspecific competition and the pace of the
		climate shift, and it is the more restrictive of the two that determines
		the observed speed.}
	\end{remark}
	
	\begin{remark}
		The two conditions \eqref{inferior} and \eqref{superior} do not cover the entire parameter space; \rev{two further regimes remain, and they are of a completely different nature.}

		\rev{The first is the \emph{weak competition} regime $\dfrac{c_1}{c_2} < \dfrac{a_1}{a_2} < \dfrac{b_1}{b_2}$, in which the kinetic system has a globally attracting coexistence state. There the invasion, when successful, drives the system to a genuine
		two-species coexistence state rather than to $(a_1/b_1,0)$. This changes the
		structure of the semi-waves that govern the front, so that a separate
		analysis is needed, and we do not carry it out here.}

		\rev{The second is the \emph{strong competition} regime $\dfrac{b_1}{b_2} < \dfrac{a_1}{a_2} < \dfrac{c_1}{c_2}$, in which the kinetic system is \emph{bistable}: both semi-trivial states $(a_1/b_1,0)$ and $(0,a_2/c_2)$ are locally stable and the coexistence state, although positive, is a saddle, so that coexistence is \emph{not} the generic outcome; ecologically, this is the priority effect. This regime is not covered by the present paper. Two of the
		tools used here break down: the monotone iteration that identifies the state
		behind the front is unavailable, because the kinetic system has no globally
		attracting state. Moreover the travelling waves of a bistable competition
		system come with a single, possibly negative, speed instead of a family
		parametrised by $c$, so that the very definition of $c_0$ through a
		semi-wave identity breaks down.}
	\end{remark}
	
	\rev{\begin{remark}\label{rm:modelling}
		A word on the modelling assumption that only the invader feels the shifting
		environment. The function $A(x-ct)$ is not meant to describe climate change
		as an agent acting on the whole community. It describes the \emph{climatic
		range limit of the invader}, which is the quantity that moves. The invasive
		species considered here is one whose distribution is bounded by a climatic
		threshold (a winter isotherm, a frost-free period, a growing-degree-day
		requirement), and it is that threshold which advances at speed $c$, opening
		new territory. The native competitor, by contrast, is already established
		throughout $[0,\infty)$ and, over the spatial and temporal scales relevant
		to the invasion, sits well inside its own climatic envelope: its range
		limits lie outside the region under consideration, so that its intrinsic
		growth rate is effectively constant there. Moreover, the invader is at its
		expanding edge, precisely where population growth is most sensitive to a
		climatic gradient, whereas an established resident population responds
		demographically on a much slower time scale. The complementary situation,
		in which both species are subject to a shifting habitat and both expand
		through a free boundary, was treated by Lei et al.~\cite{Lei2018}.

		We also note that the analysis of the semi-waves and of the spreading speed
		is insensitive to this assumption: replacing $a_2$ by a function
		$\tilde a_2(x-ct)$ shifting at the \emph{same} speed $c$ leaves the system
		autonomous in the moving frame $\xi=x-ct$, and the arguments of
		Section~\ref{sec_4} go through with $\tilde a_2(\xi)$ in place of $a_2$.
		What does change is the long-time description of $v$, whose limiting state
		is then an $x$-dependent profile rather than a constant. The case of two
		species whose climatic edges move at different speeds is genuinely
		different, since no moving frame renders the problem autonomous, and
		remains open.
	\end{remark}}

	The paper is organized as follows. In Section~\ref{sec_2}, we state some basic results on \eqref{main} whose proofs can be done similarly to those in~\cite{MR3327894} and in the existing literature. We also give some lemmas for the proofs of our main results mentioned above. In Section~\ref{sec_3} and Section~\ref{sec_4}, we study long-time dynamics and the spreading speed of the model and give proofs of Theorems~\ref{th:inferior} and~\ref{th:superior}. Section~\ref{sec_5} provides numerical simulations to demonstrate our results.
	
	\section{Preliminaries}\label{sec_2}
	
	\subsection{Existence and uniqueness}
	
	Arguing as in \cite{MR3327894}, we have the following local existence and uniqueness
	result for problem \eqref{main}.
	
	\begin{theorem}\label{th:existence}
		For any given $\alpha \in (0, 1)$ and any $(u_0, v_0)$ satisfying \eqref{initial},
		there exists a $T > 0$ such that problem \eqref{main} admits a unique bounded
		solution
		\[
		(u, v, h)
		\in C^{(1+\alpha)/2,\,1+\alpha}(D_T)
		\times C^{(1+\alpha)/2,\,1+\alpha}(D_T^{\infty})
		\times C^{1+\alpha/2}([0,T]).
		\]
		Here $C^{(1+\alpha)/2,\,1+\alpha}$ denotes the parabolic H\"{o}lder
		space of functions that are H\"{o}lder continuous of exponent $\tfrac{1+\alpha}{2}$
		in $t$ and of exponent $1+\alpha$ in $x$ (in the sense of
		\cite[Chapter~IV]{MR241821}).
		Moreover,
		\[
		\|u\|_{C^{(1+\alpha)/2,\,1+\alpha}(D_T)}
		+ \|v\|_{C^{(1+\alpha)/2,\,1+\alpha}(D_T^{\infty})}
		+ \|h\|_{C^{1+\alpha/2}([0,T])} \leq C,
		\]
		where
		$D_T = \{(t,x) \in \mathbb{R}^2 : t \in [0,T],\, x \in [0, h(t)]\}$,
		$D_T^{\infty} = \{(t,x) \in \mathbb{R}^2 : t \in [0,T],\, x \in [0, +\infty)\}$,
		and $C$ and $T$ depend only on $\alpha, h_0, \|u_0\|_{C^2([0,h_0])},
		\|v_0\|_{C^2([0,\infty))}$, on $a_0, a_1, a_2, b_1, b_2, c_1, c_2$\rev{, $d_1$, $d_2$, $c$ and on the Lipschitz constants of $A$ and $\mu$}.
	\end{theorem}
	
	The solution obtained in Theorem~\ref{th:existence} can be extended to all $t > 0$, yielding the following global estimates.
	
	\begin{theorem}\label{th:bounds}
		Problem \eqref{main} admits a unique global and uniformly bounded solution $(u,v,h)$ defined for all $t > 0$. Specifically, there exist positive constants $M_1$ and $M_2$ such that
		\[
		0 < u(t,x) \leq M_1
		\quad\text{for }(t, x) \in (0, +\infty) \times [0, h(t)),
		\]
		\[
		0 < v(t,x) \leq M_2
		\quad\text{for }(t, x) \in (0, +\infty) \times [0, +\infty).
		\]
		Moreover, there exists a constant $M_3$ such that
		\[
		0 < h'(t) \leq M_3 \quad\text{for } t \in (0,+\infty).
		\]
		\rev{One may take $M_1=\max\{\|u_0\|_{L^\infty},\,a_1/b_1\}$ and
		$M_2=\max\{\|v_0\|_{L^\infty},\,a_2/c_2\}$; in particular
		\[
		\limsup_{t\to+\infty}u(t,x)\le \frac{a_1}{b_1},\qquad
		\limsup_{t\to+\infty}v(t,x)\le \frac{a_2}{c_2}
		\qquad\text{uniformly for }x\ge0,
		\]
		since $A(\cdot)\le a_1$ and since $u$ and $v$ are subsolutions of the
		corresponding logistic equations.}
		Furthermore, problem \eqref{main} does not have any unbounded solution.
	\end{theorem}
	
	Since the proof of Theorem~\ref{th:bounds} is similar to that of
	\cite[Theorem~2.5]{MR3327894}, we omit the details.
	
	\subsection{Comparison principle}
	
	Similarly to Lemma~2.6 in \cite{MR3327894}, we can establish the following comparison principle, which plays a crucial role in our subsequent analysis.
	
	\begin{theorem}\label{th:comparison}
		Let $(u,v,h)$ be the unique bounded solution of \eqref{main}. Let $T \in (0,+\infty)$ and
		\begin{align*}
			&\underline{h}, \overline{h} \in C^1([0,T]),\\
			&\underline{u} \in C(\overline{U_T}) \cap C^{1,2}(U_T)
			\text{ with } U_T := \{ (t,x) \in \mathbb{R}^2 : t \in (0, T],\, x \in (0, \underline{h}(t))\},\\
			&\overline{u} \in C(\overline{V_T}) \cap C^{1,2}(V_T)
			\text{ with } V_T := \{ (t,x) \in \mathbb{R}^2 : t \in (0, T],\, x \in (0, \overline{h}(t))\},\\
			&\underline{v}, \overline{v}
			\in (L^{\infty} \cap C)([0,T] \times [0,\infty))
			\cap C^{1,2}((0,T] \times [0,\infty)),\\
			&\underline{v}, \overline{v}, \underline{u}, \overline{u} \ge 0,
			~ \underline{h}(0) > 0.
		\end{align*}
		
		Assume that
		\[
		\begin{cases}
			\overline{u}_t \ge d_1\overline{u}_{xx}
			+ (A(x - ct) - b_1\overline{u} - c_1 \underline{v}) \overline{u},
			& 0 < t \le T,~ 0 < x < \overline{h}(t), \\
			\underline{v}_t \le d_2\underline{v}_{xx}
			+ (a_2 - b_2\overline{u} - c_2\underline{v})\underline{v},
			& 0 < t \le T,~ 0 < x < +\infty, \\
			\overline{u}(t, x) = 0,
			& 0 < t \le T, ~ \overline{h}(t) \leq x < +\infty, \\
			\overline{h}'(t) \ge -\mu (A(\overline{h}(t) - ct))\overline{u}_x(t, \overline{h}(t)),
			& 0 < t \le T, \\
			\overline{h}(0) \ge h_0,~ \overline{u}(0, x) \ge u_0(x),
			& 0 \leq x \leq h_0, \\
			\underline{v}(0, x) \le v_0(x),
			& 0 \le x < +\infty,
		\end{cases}
		\]
		and either
		\[
		\overline{u}_x(t, 0) \le 0 \le \underline{v}_x(t, 0) \quad 0 < t \le T
		\]
		or
		\[
		u(t,0) \le \overline{u}(t,0) \text{ and } v(t,0) \ge \underline{v}(t,0)  \quad 0 < t \le T,
		\]
		then
		\[
		h(t) \le \overline{h}(t) \text{ in } (0,T],\quad
		u(t,x) \le \overline{u}(t,x),~ v(t,x) \ge \underline{v}(t,x)
		\text{ for } (t,x) \in (0,T] \times [0,+\infty).
		\]
		
		Similarly, if
		\[
		\begin{cases}
			\underline{u}_t \le d_1\underline{u}_{xx}
			+ (A(x - ct) - b_1\underline{u} - c_1 \overline{v}) \underline{u},
			& 0 < t \le T,~ 0 < x < \underline{h}(t), \\
			\overline{v}_t \ge d_2\overline{v}_{xx}
			+ (a_2 - b_2\underline{u} - c_2\overline{v})\overline{v},
			& 0 < t \le T,~ 0 < x < +\infty, \\
			\underline{u}(t, x) = 0,
			& 0 < t \le T, ~ \underline{h}(t) \leq x < +\infty, \\
			\underline{h}'(t) \le -\mu (A(\underline{h}(t) - ct))\underline{u}_x(t, \underline{h}(t)),
			& 0 < t \le T, \\
			\underline{h}(0) \le h_0,~ \underline{u}(0, x) \le u_0(x),
			& 0 \leq x \leq h_0, \\
			\overline{v}(0, x) \ge v_0(x),
			& 0 \le x < +\infty.
		\end{cases}
		\]		
		and either
		\[
		\underline{u}_x(t, 0) \ge 0 \ge \overline{v}_x(t, 0) \quad 0 < t \le T
		\]
		or
		\[
		u(t,0) \ge \underline{u}(t,0) \text{ and } v(t,0) \le \overline{v}(t,0)  \quad 0 < t \le T,
		\]
		then
		\[
		h(t) \ge \underline{h}(t) \text{ in } (0,T],\quad
		u(t,x) \ge \underline{u}(t,x),~ v(t,x) \le \overline{v}(t,x)
		\text{ for } (t,x) \in (0,T] \times [0,+\infty).
		\]
	\end{theorem}
	
	\begin{remark}
		A distinctive feature of the two-species system \eqref{main} is that the upper solution and lower solution are \emph{interleaved}: the upper solution for $u$
		is paired with the lower solution for $v$, and vice versa. This reflects the
		competitive interaction: an upper bound on $u$ yields a lower bound on $v$, and
		a lower bound on $u$ yields an upper bound on $v$. The pair $(\overline{u},
		\underline{v}, \overline{h})$ is called an \textit{upper solution} and $(\underline{u},
		\overline{v}, \underline{h})$ is called a \textit{lower solution} to \eqref{main}.
	\end{remark}
	
	\section{Long-time dynamics of the competitive systems}\label{sec_3}
	
	\subsection{Invasion of an inferior competitor}
	In this subsection, we deal with the situation where the invasive species is an inferior competitor. In other words, we assume that \eqref{inferior} holds.
	
	\begin{proof}[Proof of Theorem~\ref{th:inferior}]
		Since $v(t,x)$ satisfies
		\[
		\begin{cases}
			v_t - d_2 v_{xx} \le (a_2 - c_2 v)v, & t>0, ~ 0\le x<+\infty,\\
			v(0,x) = v_0(x) \ge 0,                & 0\le x<+\infty,
		\end{cases}
		\]
		we have $\limsup_{t\to+\infty} v(t,\cdot) \le \frac{a_2}{c_2}$ uniformly in
		$[0, +\infty)$.
		
		\rev{Since $A(\cdot)\le a_1$ and $\mu(A(\cdot))\le\mu(a_1)$, and since
		$u_x(t,h(t))<0$, the triple $(u,v,h)$ is a lower solution of
		\eqref{main2} in the sense of Theorem~\ref{th:comparison} (with the first
		alternative at $x=0$, both derivatives being zero there).}
		Let $(\tilde{u},\tilde{v},\tilde{h})$ denote the unique solution of \eqref{main2}, then the comparison principle gives that
		$u \le \tilde{u}$, $v \ge \tilde{v}$ and $h \le \tilde{h}$.
		
		On the other hand, by \cite[Theorem~3.1]{MR3327894}, we know that
		\[
		\lim_{t\to\infty} (\tilde{u}(t,\cdot), \tilde{v}(t,\cdot))
		= \left(0, \frac{a_2}{c_2}\right) \text{ in } L^\infty_{\rm loc}([0,+\infty)).
		\]
		\rev{Together with $0<u\le\tilde u$ this gives $u(t,\cdot)\to0$ in
		$L^\infty_{\rm loc}([0,+\infty))$, while $v\ge\tilde v$ yields
		$\liminf_{t\to\infty}v\ge a_2/c_2$ in $L^\infty_{\rm loc}([0,+\infty))$;
		combined with the upper bound $\limsup_{t\to\infty}v\le a_2/c_2$
		established above, we obtain}
		\[
		\lim_{t\to\infty} (u(t,\cdot), v(t,\cdot))
		= \left(0, \frac{a_2}{c_2}\right) \text{ in } L^\infty_{\rm loc}([0,+\infty)).
		\]
		
		Now suppose that $\inf_{x\in[0,+\infty)} v_0(x) > 0$. By
		\cite[Theorem~3.3]{MR3327894}, we have $\lim_{t\to\infty} \tilde{h}(t)
		< +\infty$ and $\lim_{t\to\infty} \tilde{u}(t,x) = 0$,
		$\lim_{t\to\infty} \tilde{v}(t,x) = \frac{a_2}{c_2}$ uniformly in
		$[0, \infty)$. As a consequence, $\lim_{t\to\infty} h(t) < +\infty$ and
		$\lim_{t\to\infty} u(t,x) = 0$, $\lim_{t\to\infty} v(t,x) =
		\frac{a_2}{c_2}$ uniformly in $[0, \infty)$.
	\end{proof}
	
	\subsection{Invasion of a superior competitor}
	In what follows, we assume that $u$ is a superior competitor, i.e., \eqref{superior}
	holds. By Theorem~\ref{th:bounds},
	\[
	h_\infty := \lim_{t\to\infty} h(t)
	\]
	exists and $h_\infty \in (h_0, +\infty]$.
	
	\begin{proposition}\label{prop:vanishing}
		If $h_\infty < +\infty$, then $h_\infty \le
		\frac{\pi}{2}\sqrt{\frac{d_1}{a_1-a_2c_1/c_2}}$ and
		$\lim_{t\to\infty} \|u(t,\cdot)\|_{C([0,h(t)])} = 0$,
		$\lim_{t\to\infty} v(t,x) = \frac{a_2}{c_2}$ uniformly in any compact subset
		of $[0, \infty)$.
	\end{proposition}
	
	\begin{proof}
		Let $T_0 = \frac{h_\infty}{c}$. We have $h(t) < ct$ for $t>T_0$. Hence for
		$t>T_0$,
		\[
		A(x-ct) = a_1 \text{ for } x \in [0,h(t)]
		\quad\text{ and }\quad
		\mu(A(h(t)-ct)) = \mu(a_1).
		\]
		This implies that $(t,x) \mapsto (u(t+T_0, x), v(t+T_0, x), h(t+T_0))$ solves
		\eqref{main2}, where $u_0(\cdot)$, $v_0(\cdot)$ and $h_0$ are replaced by
		$u(T_0,\cdot)$, $v(T_0,\cdot)$ and $h(T_0)$, respectively\rev{. These new
		initial functions again satisfy \eqref{initial}, since $u(T_0,\cdot)>0$ on
		$[0,h(T_0))$, $u_x(T_0,0)=u(T_0,h(T_0))=0$, $v(T_0,\cdot)>0$ and
		$v_x(T_0,0)=0$}. \rev{Since the free boundary of this shifted triple still
		converges to the finite limit $h_\infty$, vanishing occurs for
		\eqref{main2}, and the conclusion} follows from Theorem~4.3 in \cite{MR3327894}.
	\end{proof}
	
	\begin{proposition}\label{prop:spreading}
		If $h_\infty = +\infty$, then $\lim_{t\to\infty} (u(t,x), v(t,x)) =
		\left(\frac{a_1}{b_1},0\right)$ uniformly in any compact subset of
		$[0, \infty)$.
	\end{proposition}
	
	\begin{proof}
		Thanks to the comparison principle, we have
		\[
		\limsup_{t\to+\infty}u(t,x) \leq \frac{a_1}{b_1}
		\text{ uniformly for } x \in [0,+\infty),
		\]
		\[
		\limsup_{t\to+\infty}v(t,x) \leq \frac{a_2}{c_2}
		\text{ uniformly for } x \in [0,+\infty).
		\]
		Therefore, for $\varepsilon = (\frac{a_1}{c_1} - \frac{a_2}{c_2})/2 > 0$, there
		exists $T_1 > 0$ such that $v(t,x) \leq \frac{a_2}{c_2} + \varepsilon$ for
		$t \geq T_1, x \in [0, +\infty)$. Then $u$ satisfies
		\begin{equation}\label{eq1prospread}
			\begin{cases}
				u_t - d_1u_{xx} \geq (B(x - ct) - b_1u)u,
				& t > T_1,~ 0 \leq x < h(t),\\
				u_x(t,0) = 0,~ u(t, h(t)) = 0,
				& t > T_1,\\
				h'(t) \geq -\mu(a_0)u_x(t,h(t)),
				& t > T_1,\\
				u(T_1,x) > 0,
				& 0 \leq x < h(t),
			\end{cases}
		\end{equation}
		where \rev{$B(\xi) := A(\xi) - c_1\bigl(\frac{a_2}{c_2}+\varepsilon\bigr)$.
		By the choice of $\varepsilon$ one has $a_1-c_1\frac{a_2}{c_2}=2c_1\varepsilon$, whence}
		\[
		B(\xi) \rev{{}= A(\xi) - a_1 + \varepsilon c_1}
		= \begin{cases}
			\varepsilon c_1 & \text{ if } \xi \leq 0,\\
			a_0 - a_1 + \varepsilon c_1      & \text{ if } \xi \geq l_0.
		\end{cases}
		\]
		\rev{Thus $B$ satisfies the same structural assumptions as $A$, with $a_1$
		replaced by $\varepsilon c_1>0$ and $a_0$ replaced by
		$a_0-a_1+\varepsilon c_1<0$.}
		
		Next, we prove that for any $l > \frac{\pi}{2}\sqrt{\frac{d_1}{\varepsilon c_1}}$,
		there exists $T_l > \max\{T_1, \frac{l}{c}\}$ such that
		\begin{equation}\label{eq2prospread}
			u(t,x) \geq \frac{\varepsilon c_1}{2b_1}
			\text{ for } t \geq T_l,~ 0 \leq x \leq l.
		\end{equation}
		Indeed, since $h(t) \to \infty$, there exists $T_2 > T_1$ such that
		$h(T_2) \ge l$. It follows from the comparison principle that $h(t) \ge \rev{\underline{h}_l(t)}$ for $t \in (T_2, \infty)$ and $u(t, x) \ge \rev{\underline{w}_l(t,x)}$ for $t > T_2$ and $x \in [0,\rev{\underline{h}_l(t)}]$, where $\rev{(\underline{w}_l, \underline{h}_l)}$ is the solution of
		the following free boundary problem
		\begin{equation}\label{eq3prospread}
			\rev{\begin{cases}
				(\underline{w}_l)_t - d_1(\underline{w}_l)_{xx}
				= (B(x - ct) - b_1\underline{w}_l)\underline{w}_l,
				& t > T_2,~ 0 \leq x < \underline{h}_l(t),\\
				(\underline{w}_l)_x(t,0) = 0,
				~ \underline{w}_l(t, \underline{h}_l(t)) = 0,
				& t > T_2,\\
				\underline{h}_l'(t) = -\mu(a_0)(\underline{w}_l)_x(t,\underline{h}_l(t)),
				& t > T_2,\\
				\underline{w}_l(T_2,x) = u(T_2,x),
				~ \underline{h}_l(T_2) = h(T_2),
				& 0 \leq x \le h(T_2).
			\end{cases}}
		\end{equation}
		Since $h(T_2) \ge l > \frac{\pi}{2}\sqrt{\frac{d_1}{\varepsilon c_1}}$, it
		follows from \cite[Theorems~1.2--1.5]{HU20205931}\rev{, applied to
		\eqref{eq3prospread} with initial time $T_2$ and with the shifted
		environment $\xi\mapsto B(\xi-cT_2)$, which satisfies the same structural
		assumptions,} that $\rev{\underline{h}_l(t)} \to
		\infty$ and $\rev{\underline{w}_l(t,x)} \to \frac{\varepsilon c_1}{b_1}$ as
		$t \to \infty$ uniformly in $[0,l]$, which implies that \eqref{eq2prospread} is
		true for some $T_l$. Therefore, \rev{recalling that $T_l>l/c$, so that
		$A(x-ct)=a_1$ for all $t\ge T_l$ and $0\le x\le l$,} $(u,v)$ satisfies
		\begin{equation}\label{eq4prospread}
			\begin{cases}
				u_t - d_1u_{xx} = (a_1 - b_1u - c_1v)u,
				& t > T_l,~ 0 \leq x < l,\\
				v_t - d_2v_{xx} = (a_2 - b_2u - c_2v)v,
				& t > T_l,~ 0 \leq x < l,\\
				u_x(t,0) = v_x(t, 0) = 0,
				& t > T_l,\\
				u(t,x) \geq \frac{\varepsilon c_1}{2b_1},
				~ v(t,x) \leq \frac{a_2}{c_2} + \varepsilon,
				& t \geq T_l,~ 0 \leq x \leq l.
			\end{cases}
		\end{equation}
		%
		%
		It follows from the theory of monotone dynamical systems (see, e.g., \cite[Chapter 4]{MR1319817} or
		\cite[Section 2]{MR1100011}) that
		\[
		\liminf_{t \to +\infty} u(t,x) \geq \underline{u}_l(x)
		\text{ and }
		\limsup_{t \to +\infty} v(t,x) \leq \overline{v}_l(x)
		\quad \text{ in } [0,l],
		\]
		where $(\underline{u}_l, \overline{v}_l)$ satisfies
		\begin{equation}\label{eq5prospread}
			\begin{cases}
				-d_1(\underline{u}_l)_{xx}
				= (a_1 - b_1\underline{u}_l - c_1\overline{v}_l)\underline{u}_l,
				& 0 \leq x < l,\\
				-d_2(\overline{v}_l)_{xx}
				= (a_2 - b_2\underline{u}_l - c_2\overline{v}_l)\overline{v}_l,
				& 0 \leq x < l,\\
				(\underline{u}_l)_x(0) = (\overline{v}_l)_x(0) = 0,\\
				\underline{u}_l(l) = \frac{\varepsilon c_1}{2b_1},
				~ \overline{v}_l(l) = \frac{a_2}{c_2} + \varepsilon.
			\end{cases}
		\end{equation}
		Letting $l \to \infty$, by classical elliptic regularity theory and a diagonal
		procedure, we conclude that $(\underline{u}_l,\overline{v}_l)$ converges
		uniformly on any compact set of $[0, \infty)$ to the solution
		$(\underline{u}_{\infty}, \overline{v}_{\infty})$ of
		\begin{equation}\label{eq6prospread}
			\begin{cases}
				-d_1(\underline{u}_{\infty})_{xx}
				= (a_1 - b_1\underline{u}_{\infty} - c_1\overline{v}_{\infty})
				\underline{u}_{\infty},
				& 0 \leq x < \infty,\\
				-d_2(\overline{v}_{\infty})_{xx}
				= (a_2 - b_2\underline{u}_{\infty} - c_2\overline{v}_{\infty}) \overline{v}_{\infty},
				& 0 \leq x < \infty,\\
				(\underline{u}_{\infty})_x(0) = (\overline{v}_{\infty})_x(0) = 0,\\
				\underline{u}_{\infty}(x) \geq \frac{\varepsilon c_1}{2b_1},
				~ \overline{v}_{\infty}(x) \leq \frac{a_2}{c_2} + \varepsilon.
			\end{cases}
		\end{equation}
		Next, we show that $\underline{u}_{\infty} \equiv \frac{a_1}{b_1}$ and
		$\overline{v}_{\infty} \equiv 0$. In fact, consider the ODE system
		\[
		\begin{cases}
			z_t = (a_1 - b_1z - c_1w)z, & t > 0,\\
			w_t = (a_2 - b_2z - c_2w)w, & t > 0,\\
			z(0) = \frac{\varepsilon c_1}{2b_1},~ w(0) = \frac{a_2}{c_2} + \varepsilon.
		\end{cases}
		\]
		Since $\frac{a_1}{a_2} > \max\{ \frac{b_1}{b_2}, \frac{c_1}{c_2}\}$, it is
		well-known that $(z,w) \to (\frac{a_1}{b_1},0)$ as $t \to \infty$. Hence, the
		solution $(Z(t,x), W(t,x))$ of the problem
		\[
		\begin{cases}
			Z_t - d_1Z_{xx} = (a_1 - b_1Z - c_1W)Z,
			& t > 0, ~0 \leq x \leq l,\\
			W_t - d_2W_{xx} = (a_2 - b_2Z - c_2W)W,
			& t > 0, ~0 \leq x \leq l,\\
			Z_x(t,0) = W_x(t, 0) = 0,
			& t > 0,\\
			Z(0,x) = \frac{\varepsilon c_1}{2b_1},
			~ W(0,x) = \frac{a_2}{c_2} + \varepsilon,
			& x \geq 0
		\end{cases}
		\]
		satisfies $(Z,W) \to (\frac{a_1}{b_1},0)$ uniformly in $[0, \infty)$ as
		$t \to +\infty$. It follows from the comparison principle that
		$\underline{u}_{\infty}(x) \geq Z(t,x)$ and $\overline{v}_{\infty}(x) \leq
		W(t,x)$ for $t > 0$. Therefore, $\underline{u}_{\infty} = \frac{a_1}{b_1}$ and
		$\overline{v}_{\infty} = 0$. We thus obtain $\liminf_{t \to +\infty} u(t,x)
		\geq \frac{a_1}{b_1}$ and $\limsup_{t \to +\infty}v(t,x) \leq 0$ uniformly in
		$[0,l]$\rev{; combined with the upper bounds
		$\limsup_{t\to+\infty}u\le a_1/b_1$ and $v>0$ recalled at the beginning of
		the proof, this gives the assertion}. Then, $\lim_{t\to+\infty}u(t,x) = \frac{a_1}{b_1}$ and
		$\lim_{t\to+\infty}v(t,x) = 0$ uniformly in any compact subset of $[0, \infty)$
		and this completes the proof.
	\end{proof}
	
	\section{Spreading speed of the invasive species}\label{sec_4}
	In this section we assume that \eqref{superior} holds, that spreading happens
	for $u$, i.e. $h_\infty=+\infty$, and that
	$\liminf_{x\to\infty}v_0(x)>0$. The last hypothesis is used from
	Lemma~\ref{lem:limsup} onwards, where the native population has to be bounded away
	from $0$ far ahead of the front. Our aim is the following.
	\begin{proposition}\label{prop:speed}
		Suppose that \eqref{superior} holds, $h_\infty = +\infty$ and
		\[
		\liminf_{x\to\infty} v_0(x) > 0.
		\]
		Then
		\[
		\lim_{t\to\infty} \frac{h(t)}{t} = \min\{c, c_0\}.
		\]
	\end{proposition}
	
	\subsection{Semi-wave solutions}
	We recall the following result on semi-wave solutions for \eqref{main2}.
	\begin{theorem}[Theorem 1.3 and Remark 2.10 in \cite{MR3609207}]\label{th:semi-wave0} The problem
		\begin{equation}\label{semi-wave0}
			\begin{cases}
				-c \psi' - d_1 \psi'' = (a_1 - b_1\psi - c_1\varphi)\psi, \quad-\infty < x < 0,\\
				-c \varphi' - d_2 \varphi'' = (a_2 - b_2\psi - c_2\varphi)\varphi, \quad -\infty < x < +\infty,\\
				\psi(-\infty) = \frac{a_1}{b_1}, ~ \psi(x)=0 \text{ for } x\ge 0, ~ \psi'(x)<0 \text{ for } x\le 0,\\	
				\varphi(-\infty) = 0, ~ \varphi(+\infty)=\frac{a_2}{c_2}, ~ \varphi'(x)>0 \text{ for } x\in\mathbb{R}
			\end{cases}
		\end{equation}
		has a unique solution $(\Psi_c, \Phi_c) \in [C(\mathbb{R})\cap C^2((-\infty, 0])] \times C^2(\mathbb{R})$ for each $c\in[0,c^*)$ and it has no such solution for $c\ge c^*$. Furthermore, $c\mapsto \Psi_c'(0)$ is increasing on $[0,c^*)$, there exists a unique $c_0 \in (0,c^*)$ such that
		\[
		-\mu(a_1) \Psi_{c_0}'(0) = c_0,
		\]
		and $c_0$ depends continuously on the parameters of the problem.
	\end{theorem}
	
	\rev{\begin{remark}\label{rm:cstar}
		The threshold $c^*$ in Theorem \textup{\ref{th:semi-wave0}} is the minimal
		speed of the travelling waves of the homogeneous competition system
		\eqref{main2}. Undoing the non-dimensionalisation used in
		\cite{MR3609207} (namely
		$\hat u=\frac{b_1}{a_1}u\bigl(t/a_2,\sqrt{d_2/a_2}\,x\bigr)$,
		$\hat v=\frac{c_2}{a_2}v\bigl(t/a_2,\sqrt{d_2/a_2}\,x\bigr)$, under which a
		speed $s$ becomes $c=\sqrt{a_2d_2}\,s$), \cite[Proposition 1.2]{MR3609207}
		gives the explicit two-sided bound
		\[
		2\sqrt{d_1\Bigl(a_1-\frac{a_2c_1}{c_2}\Bigr)}\ \le\ c^*\ \le\ 2\sqrt{d_1a_1},
		\]
		both quantities being positive under \eqref{superior}. In particular
		$0<c_0<c^*\le2\sqrt{d_1a_1}$: the speed produced by the free boundary model
		always stays strictly below the Fisher--KPP speed $2\sqrt{d_1a_1}$ of the
		Cauchy problem for $u$ alone. We stress that $2\sqrt{d_1a_1}$ is
		\emph{not} the speed the invader would have in the absence of the native
		species: the latter is the free boundary speed $c_0^{\rm sing}$ of the
		single species problem of \cite{MR2607347}, which is itself strictly
		smaller than $2\sqrt{d_1a_1}$. The comparison between $c_0$ and
		$c_0^{\rm sing}$ is discussed in Remark \textup{\ref{rm:speed}}.
	\end{remark}}

	To prove Proposition \ref{prop:speed}, we need the following result on an auxiliary elliptic system, which supplies a semi-wave to \eqref{main} with speed $c$.
	\begin{theorem}\label{th:semi-wave} Suppose that $0<c\le c_0$. Then we have the following conclusions:
		\begin{itemize}
			\item[(i)] For any $L\ge0$, the problem
			\begin{equation}\label{semi-wave}
				\begin{cases}
					-c \psi' - d_1 \psi'' = (A(x) - b_1\psi - c_1\varphi)\psi, \quad-\infty < x < L,\\
					-c \varphi' - d_2 \varphi'' = (a_2 - b_2\psi - c_2\varphi)\varphi, \qquad -\infty < x < +\infty,\\
					\psi(-\infty) = \frac{a_1}{b_1}, ~ \psi(x)=0 \text{ for } x\ge L, ~ \psi'(x)<0 \text{ for } x\le L,\\	
					\varphi(-\infty) = 0, ~ \varphi(+\infty)=\frac{a_2}{c_2}, ~ \varphi'(x)>0 \text{ for } x\in\mathbb{R}
				\end{cases}			
			\end{equation}
			has a unique solution $(\psi_L,\varphi_L) \in [C(\mathbb{R})\cap C^2((-\infty, L])] \times C^2(\mathbb{R})$.
			\item[(ii)] The mapping $L \mapsto \psi'_L(L)$ is strictly increasing on $[0, +\infty)$.
			\item[(iii)] There exists a unique $L_0 \geq 0$ such that $-\mu(A(L_0))\psi'_{L_0}(L_0) = c$. Moreover, $L_0 = 0$ if and only if $c = c_0$, and in such a case, we have $(\psi_0,\varphi_0) \equiv (\Psi_{c_0}, \Phi_{c_0})$.
		\end{itemize}
	\end{theorem}
	
	\begin{proof} 
		Let $(\Psi_c,\Phi_c)$ be the unique solution of \eqref{semi-wave0}.
		
		\textbf{(i) Existence and uniqueness for a fixed $L\ge0$.}
		
		\textit{Step 1. Construction of an ordered pair of upper/lower solutions.}
		
		Define
		\[
		\overline\psi(x)=\Psi_c(x-L),\qquad \underline\varphi(x)=\Phi_c(x-L).
		\]
		Because $A(x)\le a_1$ for all $x$, we have for $x<L$
		\[
		\begin{aligned}
			-c\overline\psi'-d_1\overline\psi'' 
			&= \bigl(a_1-b_1\overline\psi-c_1\Phi_c(x-L)\bigr)\overline\psi \\
			&= \bigl(a_1-b_1\overline\psi-c_1\underline\varphi\bigr)\overline\psi \\
			&\ge \bigl(A(x)-b_1\overline\psi-c_1\underline\varphi\bigr)\overline\psi .
		\end{aligned}
		\]
		For $\underline\varphi(x)=\Phi_c(x-L)$, the $\varphi$-equation gives
		\[
		-c\underline\varphi'-d_2\underline\varphi'' = \bigl(a_2-b_2\overline\psi-c_2\underline\varphi\bigr)\underline\varphi.
		\]
		Hence $(\overline\psi,\underline\varphi)$ is an upper solution of the problem
		\begin{equation}\label{sw1}
			\begin{cases}
				-c \psi' - d_1 \psi'' = (A(x) - b_1\psi - c_1\varphi)\psi, \quad-\infty < x < L,\\
				-c \varphi' - d_2 \varphi'' = (a_2 - b_2\psi - c_2\varphi)\varphi, \qquad -\infty < x < +\infty,\\
				\psi(x)=0 \text{ for } x\ge L.
			\end{cases}
		\end{equation}
		
		Similarly, we can easily check that $(\underline\psi,\overline\varphi) := (\Psi_c, \Phi_c)$ is a lower solution of \eqref{sw1}. Note that $\underline\psi$ has a corner at \rev{$x=0$}, making it a weak lower solution, which is admissible for the monotone iteration scheme. Clearly, by the monotonicity of $\Psi_c$ and $\Phi_c$, we have $\underline\psi\le\overline\psi$
		and $\underline\varphi\le\overline\varphi$.
		
		\textit{Step 2. Existence via monotone iteration on bounded domains.}
		For any $n\in\mathbb{N}$ with $n>L$, we consider the truncated domain $(-n,L)$ for $\psi$ and
		$(-n,n)$ for $\varphi$.  On these bounded domains we impose the boundary conditions
		\[
		\psi(-n)=\overline\psi(-n),\quad \psi(L)=0,\qquad
		\varphi(-n)=\underline\varphi(-n),\quad \varphi(n)=\overline\varphi(n).
		\]
		That is, we solve the elliptic system
		\[
		\begin{cases}
			-c \psi' - d_1 \psi'' = (A(x) - b_1\psi - c_1\varphi)\psi, \quad-n < x < L,\\
			-c \varphi' - d_2 \varphi'' = (a_2 - b_2\psi - c_2\varphi)\varphi, \qquad -n < x < n,\\
			\psi(-n)=\overline\psi(-n), ~ \psi(x)=0 \text{ for } L \le x \le n,\\
			\varphi(-n)=\underline\varphi(-n), ~ \varphi(n)=\overline\varphi(n)
		\end{cases}
		\]
		using the standard monotone iteration scheme for
		quasimonotone systems (see e.g. \cite{MR3609207}).  Because
		$(\underline\psi,\overline\varphi)$ and $(\overline\psi,\underline\varphi)$ are
		ordered lower and upper solutions, the iteration converges to a solution
		$(\psi_n,\varphi_n)$ on $[-n,L]\times[-n,n]$ satisfying
		$\underline\psi\le\psi_n\le\overline\psi$ and $\underline\varphi\le\varphi_n\le\overline\varphi$.
		Standard elliptic regularity (Schauder estimates) gives uniform $C^{2,\alpha}$ bounds
		independent of $n$.  Letting $n\to\infty$ and using a diagonal argument we obtain a
		pair $(\psi_L,\varphi_L)$ defined on $(-\infty,L]\times\mathbb{R}$ that satisfies \eqref{sw1} in the classical sense and inherits the inequalities
		\begin{equation}\label{sw2}
			\underline\psi\le\psi_L\le\overline\psi, \quad \underline\varphi\le\varphi_L\le\overline\varphi.
		\end{equation}
		
		\rev{\textit{Step 3. Asymptotic expansions of $\psi_L(x)$ and $\varphi_L(x)$ as $x \to -\infty$.}}
		\rev{From the bounds \eqref{sw2} we have $\psi_L(-\infty)=a_1/b_1$,
		$\varphi_L(-\infty)=0$ and $\varphi_L(+\infty)=a_2/c_2$. Moreover,
		$0<\psi_L<a_1/b_1$ on $(-\infty,L)$ and $0<\varphi_L<a_2/c_2$ on $\mathbb{R}$.}
		To continue, we need asymptotic expansions of \(\psi(x):=\psi_L(x)\) and \(\varphi(x):=\varphi_L(x)\) as \(x \to -\infty.\)		
		A direct computation reveals that the first-order ODE system satisfied by the quadruple $(\psi, \psi', \varphi, \varphi')$ admits a critical point at \((\frac{a_1}{b_1}, 0, 0, 0)\), which is a saddle point. Therefore, by standard stable manifold theory (see, for example, \cite[Theorem 4.1, Chapter 13]{MR69338}), 
		\[
		\frac{a_1}{b_1} - \psi(x) \to 0, ~ \varphi(x) \to 0 \text{ exponentially as } x \to -\infty. 
		\]
		The equations satisfied by \(\varphi\) and \( \hat{\psi} : = \frac{a_1}{b_1} - \psi\) may be written in the form
		\begin{equation}\label{sw3}
			\begin{cases}
				c \hat\psi ' + d_1 \hat\psi '' -a_1 \hat\psi + \frac{a_1c_1}{b_1} \varphi -[A(x) - a_1](a_1/b_1 - \hat\psi) + \varepsilon_1(x) \hat\psi= 0,\\
				-c \varphi ' - d_2 \varphi '' - (a_2 - \frac{a_1b_2}{b_1}) \varphi + \varepsilon_2(x) \varphi = 0,
			\end{cases}        
		\end{equation}
		where
		\[
		\varepsilon_1(x) = b_1 \hat\psi - c_1\varphi \to 0 \text{ exponentially as } x \to -\infty,
		\] 
		\[
		\varepsilon_2(x) = -b_2 \hat\psi(x) + c_2\varphi(x) \to 0 \text{ exponentially as } x \to -\infty,
		\]
		and
		\[
		[A(x) - a_1](a_1/b_1 - \hat\psi) \equiv 0 \text{ for } x \le 0.
		\]
		Set 
		\[
		\eta_1 := \frac{-c + \sqrt{c^2 -4d_2(a_2 - \frac{a_1b_2}{b_1})}}{2d_2},\quad \eta_2 := \frac{-c - \sqrt{c^2 -4d_2(a_2 - \frac{a_1b_2}{b_1})}}{2d_2},        
		\]
		and
		\[
		\tau(y) = -d_1y^2 - cy + a_1,~ \lambda_1 : = \frac{-c + \sqrt{c^2 + 4a_1d_1}}{2d_1}, ~ \lambda_2 : = \frac{-c - \sqrt{c^2 + 4a_1d_1}}{2d_1}.
		\]
		Define 
		\[
		q_1(x) := e^{\eta_1 x}, \quad q_2(x) : = e^{\eta_2 x}
		\]
		and
		\[
		p_1(x) := \begin{cases}
			\frac{a_1c_1}{b_1\tau(\eta_1)}e^{\eta_1 x}, &  \eta_1 \neq \lambda_1,\\
			\frac{a_1c_1}{b_1\tau'(\eta_1)}xe^{\eta_1 x}, &  \eta_1 = \lambda_1,\\
		\end{cases} \quad
		p_2(x):= \begin{cases}
			\frac{a_1c_1}{b_1\tau(\eta_2)}e^{\eta_2 x}, &  \eta_2 \neq \lambda_2,\\
			\frac{a_1c_1}{b_1\tau'(\eta_2)}xe^{\eta_2 x}, &  \eta_2 = \lambda_2,\\
		\end{cases} \quad
		p_3(x) := e^{\lambda_1 x}; \quad p_4(x) := e^{\lambda_2 x}.
		\]
		It is easy to see that 
		\[
		\mathbf{P}_1 := (p_1, q_1), ~ \mathbf{P}_2 := (p_2, q_2),~\mathbf{P}_3 := (p_3, 0),~\mathbf{P}_4 := (p_4, 0)
		\]
		are linearly independent solutions of the linear system
		\[
		\begin{cases}
			d_1 p'' + c p ' -a_1 p + \frac{a_1c_1}{b_1} q = 0,\\
			d_2 q '' + c q ' + (a_2 - \frac{a_1b_2}{b_1}) q  = 0.
		\end{cases}  
		\]            
		
		For $x \le 0$, by applying Theorem 8.1 in Chapter 3 of \cite{MR69338} (for the case $\eta_1 \neq \lambda_1$ and $\eta_2 \neq \lambda_2$) or the corresponding versions in \cite{MR658490} for the degenerate cases, we conclude that system \eqref{sw3} admits four linearly independent solutions \(\hat{\mathbf{P}}_i\) satisfying 
		\[
		\hat{\mathbf{P}}_i(x) = (1 + o(1))\mathbf{P}_i (x) \text{ as } x \to -\infty, ~ i = 1,2,3,4. 
		\]
		Since \((\hat\psi, \varphi)\) solves \eqref{sw3}, there exist constants \(\alpha_i ~ (i=1,2,3,4)\) such that 
		\[
		(\hat\psi, \varphi) = \sum_{i=1}^{4} \alpha_i \mathbf{P}_i.
		\]
		
		Thanks to \(\eta_2 <0, \lambda_2 <0\) and \(\varphi(-\infty) = 0 = \hat\psi(-\infty),\) we have \(\alpha_2 = \alpha_4 = 0.\) Since \(\hat\psi(x) >0\) and \(\varphi(x)>0,\) we deduce that \(\alpha_1 > 0\) and in the case \(\eta_1 > \lambda_1\), we have \(\alpha_3 > 0.\) \rev{Indeed, $\tau$ is a downward parabola vanishing exactly at $\lambda_2<0<\lambda_1$, so $\tau(\eta_1)<0$ whenever $\eta_1>\lambda_1$. Hence $\alpha_3=0$ would force $\hat\psi<0$ near $-\infty$, a contradiction. For the same reason $\tau(\eta_1)>0$ when $0<\eta_1<\lambda_1$, and $\tau'(\eta_1)=-(2d_1\eta_1+c)<0$ when $\eta_1=\lambda_1$.} Thus, as \(x \to -\infty,\)
		\[
		\varphi(x) = \alpha_1 e^{\eta_1x}(1 + o(1)),        
		\]
		and
		\[
		\hat\psi(x) = \begin{cases}
			\alpha_1 \frac{a_1c_1}{b_1\tau(\eta_1)} e^{\eta_1x} (1 + o(1)), & \text{if } \eta_1 < \lambda_1,\\
			\rev{-}\alpha_1\frac{a_1c_1}{b_1\tau'(\eta_1)}\rev{|x|}e^{\eta_1x} (1 + o(1)),& \text{if } \eta_1 = \lambda_1,\\
			\alpha_3e^{\lambda_1x} (1 + o(1)), & \text{if } \eta_1 > \lambda_1.
		\end{cases}
		\]        
		Therefore, there exist positive constants \(K_\varphi\) and \(K_\psi\) such that, as \(x \to -\infty,\)
		\[
		\varphi(x) = K_\varphi e^{\eta_1 x}(1 + o(1)),
		\]
		and
		\[
		\psi(x) = \begin{cases}
			\frac{a_1}{b_1} - K_\psi  e^{\eta_1x} (1 + o(1)), & \text{if } \eta_1 < \lambda_1,\\
			\frac{a_1}{b_1} - K_\psi \rev{|x|}e^{\eta_1x} (1 + o(1)),& \text{if } \eta_1 = \lambda_1,\\
			\frac{a_1}{b_1} - K_\psi e^{\lambda_1x} (1 + o(1)), & \text{if } \eta_1 > \lambda_1.
		\end{cases}
		\]

		\rev{It is convenient to record this in a unified form. Set
		\[
		(\lambda_0,m):=
		\begin{cases}
			(\eta_1,0) & \text{ if } \eta_1<\lambda_1,\\
			(\eta_1,1) & \text{ if } \eta_1=\lambda_1,\\
			(\lambda_1,0) & \text{ if } \eta_1>\lambda_1 .
		\end{cases}
		\]
		Then $\lambda_0>0$, $\eta_1>0$ and $m\in\{0,1\}$ depend only on the
		coefficients $c,d_1,d_2,a_1,a_2,b_1,b_2,c_1$ of the problem \emph{and not on
		the particular solution}, and there exist positive constants
		$K_\psi,K_\varphi$ (which do depend on the solution) such that
		\begin{equation}\label{asym-unified}
			\frac{a_1}{b_1}-\psi(x)=K_\psi|x|^{m}e^{\lambda_0x}(1+o(1)),
			\qquad
			\varphi(x)=K_\varphi e^{\eta_1x}(1+o(1))
			\qquad (x\to-\infty).
		\end{equation}
		Differentiating the corresponding relations for the solutions
		$\hat{\mathbf P}_i$ of \eqref{sw3}, the same expansions hold for
		$\psi'$ and $\varphi'$ after differentiation of the right-hand sides.}

		\rev{\textit{Step 4. A sliding lemma, and monotonicity of $\psi_L$ and $\varphi_L$.}}

		\rev{We shall use the following sliding argument three times, so we state it
		once and for all. Let $\ell\ge0$, let $(\Psi,\Phi)$ solve
		\begin{equation}\label{sw-class}
			\begin{cases}
				-c \Psi' - d_1 \Psi'' = (A(x) - b_1\Psi - c_1\Phi)\Psi, & -\infty<x<\ell,\\
				-c \Phi' - d_2 \Phi'' = (a_2 - b_2\Psi - c_2\Phi)\Phi, & x\in\mathbb{R},
			\end{cases}
		\end{equation}
		with $\Psi\equiv0$ on $[\ell,\infty)$, $0<\Psi<a_1/b_1$ on $(-\infty,\ell)$,
		$\Psi(-\infty)=a_1/b_1$, $0<\Phi<a_2/c_2$, $\Phi(-\infty)=0$ and
		$\Phi(+\infty)=a_2/c_2$; and let $(\tilde\Psi,\tilde\Phi)$ satisfy the same
		conditions with $\ell$ replaced by some $\tilde\ell\le\ell$ and with the
		first equation of \eqref{sw-class} replaced by the inequality ``$\le$''
		(that is, $(\tilde\Psi,\tilde\Phi)$ is a lower solution). Assume finally
		that both pairs satisfy \eqref{asym-unified} with the same exponents
		$\lambda_0,\eta_1$ and the same $m$. Then
		\begin{equation}\label{sw-slide}
			\tilde\Psi\le\Psi \text{ on } (-\infty,\tilde\ell\,],
			\qquad
			\tilde\Phi\ge\Phi \text{ on } \mathbb{R}.
		\end{equation}}

		\rev{\emph{Proof of \eqref{sw-slide}.} For $\xi\ge0$ put
		$\tilde\Psi^\xi:=\tilde\Psi(\cdot+\xi)$ and $\tilde\Phi^\xi:=\tilde\Phi(\cdot+\xi)$.
		Since $A$ is nonincreasing, $(\tilde\Psi^\xi,\tilde\Phi^\xi)$ is again a lower
		solution, now on $(-\infty,\tilde\ell-\xi)$.}

		\rev{\emph{(a) Behaviour at $+\infty$.} For $x\ge\ell$ we have
		$\Psi(x)=\tilde\Psi^\xi(x)=0$, so both $\Phi$ and $\tilde\Phi^\xi$ solve the
		scalar equation $-c\theta'-d_2\theta''=(a_2-c_2\theta)\theta$ there. Setting
		$W:=\tilde\Phi^\xi-\Phi$, we get $d_2W''+cW'+qW\le0$ on $(\ell,\infty)$ with
		$q:=a_2-c_2(\tilde\Phi^\xi+\Phi)\to-a_2<0$ as $x\to+\infty$. Fix
		$\epsilon\in(0,\frac{a_2}{2c_2})$ and choose $R>\ell$ so large that
		$\Phi>\frac{a_2}{c_2}-\epsilon$ and $\tilde\Phi>\frac{a_2}{c_2}-\epsilon$ on
		$[R,\infty)$; since $\xi\ge0$, this gives
		$\tilde\Phi^\xi>\frac{a_2}{c_2}-\epsilon$ on $[R,\infty)$ as well, whence
		$q<-a_2+2c_2\epsilon<0$ there, \emph{uniformly in} $\xi\ge0$. Since $W(+\infty)=0$, a negative interior minimum of
		$W$ on $[R,\infty)$ is impossible: at such a point $W''\ge0$, $W'=0$ and
		$qW>0$, contradicting $d_2W''+cW'+qW\le0$. Consequently
		\begin{equation}\label{sw-tail}
			\inf_{[R,\infty)}\bigl(\tilde\Phi^\xi-\Phi\bigr)\ \ge\ \min\bigl\{0,\ \tilde\Phi^\xi(R)-\Phi(R)\bigr\}.
		\end{equation}
		Thus it suffices to control $\tilde\Phi^\xi-\Phi$ on $(-\infty,R]$.}

		\rev{\emph{(b) Strict inequalities near $-\infty$.} By \eqref{asym-unified},
		for every $\varepsilon\in(0,1)$ there is $X_\varepsilon<0$ such that, for
		$y\le X_\varepsilon$,
		\[
		(1-\varepsilon)K_\psi|y|^{m}e^{\lambda_0y}\le\frac{a_1}{b_1}-\Psi(y)\le(1+\varepsilon)K_\psi|y|^{m}e^{\lambda_0y},
		\qquad
		(1-\varepsilon)K_\varphi e^{\eta_1y}\le\Phi(y)\le(1+\varepsilon)K_\varphi e^{\eta_1y},
		\]
		and likewise for $(\tilde\Psi,\tilde\Phi)$ with the constants
		$\tilde K_\psi,\tilde K_\varphi$. If $x\le X_\varepsilon-\xi$, then both $x$
		and $x+\xi$ lie below $X_\varepsilon$ and $|x+\xi|\ge|X_\varepsilon|$, so
		\[
		\frac{\frac{a_1}{b_1}-\tilde\Psi^{\xi}(x)}{\frac{a_1}{b_1}-\Psi(x)}
		\ \ge\ \frac{(1-\varepsilon)\tilde K_\psi}{(1+\varepsilon)K_\psi}\,
		\Bigl(\frac{|X_\varepsilon|}{|X_\varepsilon|+\xi}\Bigr)^{m}e^{\lambda_0\xi},
		\qquad
		\frac{\tilde\Phi^{\xi}(x)}{\Phi(x)}
		\ \ge\ \frac{(1-\varepsilon)\tilde K_\varphi}{(1+\varepsilon)K_\varphi}\,e^{\eta_1\xi},
		\]
		and both right-hand sides tend to $+\infty$ as $\xi\to+\infty$ (recall
		$m\in\{0,1\}$ and $\lambda_0,\eta_1>0$). Hence, for every $\xi_*>0$ one may
		choose $\varepsilon$ small and $\xi_1\ge\xi_*$ so that
		\begin{equation}\label{sw-strict}
			\tilde\Psi^\xi<\Psi \quad\text{and}\quad \tilde\Phi^\xi>\Phi
			\qquad\text{ on }(-\infty,X_\varepsilon-\xi]\ \text{ for every }\xi\ge\xi_1 .
		\end{equation}
		(When $\xi$ is restricted to a compact subinterval of $(0,\infty)$, the same
		computation shows that \eqref{sw-strict} holds there for $\varepsilon$ small
		enough, with the half-line $(-\infty,X_\varepsilon-\xi]$ replaced by a fixed
		half-line.)}

		\rev{\emph{(c) The sliding set is nonempty.} Write $X:=X_\varepsilon$ with
		$\varepsilon$ as in (b). Set
		\[
		M_0:=\max_{[X,\tilde\ell]}\tilde\Psi<\frac{a_1}{b_1},
		\qquad
		N_0:=\sup_{(-\infty,R]}\Phi<\frac{a_2}{c_2},
		\]
		and fix $Z\ge R$ with $\tilde\Phi>N_0$ on $[Z,+\infty)$ and put
		$m_0:=\min_{[X,Z]}\tilde\Phi>0$. For $\xi$ large:
		\begin{itemize}
			\item if $x\in[X-\xi,\tilde\ell-\xi]$ then $x+\xi\in[X,\tilde\ell]$, so
			$\tilde\Psi^{\xi}(x)\le M_0$, while $x\le\tilde\ell-\xi$ and
			$\Psi(-\infty)=\frac{a_1}{b_1}>M_0$ give $\Psi(x)>M_0$;
			\item if $x\in[X-\xi,Z-\xi]$ then $x+\xi\in[X,Z]$, so
			$\tilde\Phi^{\xi}(x)\ge m_0$, while
			$\Phi(x)\le(1+\varepsilon)K_\varphi e^{\eta_1(Z-\xi)}<m_0$;
			\item if $x\in[Z-\xi,R]$ then $x+\xi\ge Z$, so
			$\tilde\Phi^{\xi}(x)>N_0\ge\Phi(x)$;
			\item for $x\ge R$ the inequality follows from \eqref{sw-tail}.
		\end{itemize}
		Together with \eqref{sw-strict} this shows that the set
		\[
		S:=\bigl\{\xi\ge0:\ \tilde\Psi^\sigma\le\Psi \text{ on }(-\infty,\tilde\ell-\sigma],\
		\tilde\Phi^\sigma\ge\Phi \text{ on }\mathbb{R},\ \forall\sigma\ge\xi\bigr\}
		\]
		is nonempty.}

		\rev{\emph{(d) $\bar\xi:=\inf S=0$.} By continuity $\bar\xi\in S$. Suppose
		$\bar\xi>0$, and use the parenthetical remark in (b) with
		$\xi_*=\bar\xi/2$ to obtain a fixed $X_*<0$ such that
		\begin{equation}\label{sw-unif}
			\tilde\Psi^\sigma<\Psi \ \text{ and }\ \tilde\Phi^\sigma>\Phi
			\quad\text{ on }(-\infty,X_*]\ \text{ for every }\sigma\in[\bar\xi/2,\bar\xi].
		\end{equation}
		Set $P:=\Psi-\tilde\Psi^{\bar\xi}\ge0$ on
		$(-\infty,\tilde\ell-\bar\xi]$ and $Q:=\tilde\Phi^{\bar\xi}-\Phi\ge0$ on
		$\mathbb{R}$. Passing to the cooperative variables $(\Psi,\,a_2/c_2-\Phi)$
		and using the mean value theorem, $(P,Q)$ satisfies a linear cooperative
		system of differential inequalities with bounded coefficients. Since
		$P(\tilde\ell-\bar\xi)=\Psi(\tilde\ell-\bar\xi)>0$ and, by
		\eqref{sw-unif}, $P>0$ and $Q>0$ on $(-\infty,X_*]$, the strong maximum
		principle for cooperative systems (see \cite[Lemma~2.1]{MR3609207}) gives
		$P>0$ on $(-\infty,\tilde\ell-\bar\xi]$ and $Q>0$ on $\mathbb{R}$. The maps
		$(\sigma,x)\mapsto\Psi(x)-\tilde\Psi^{\sigma}(x)$ and
		$(\sigma,x)\mapsto\tilde\Phi^{\sigma}(x)-\Phi(x)$ are continuous and, at
		$\sigma=\bar\xi$, strictly positive on the compact sets
		$[X_*,\tilde\ell-\bar\xi]$ and $[X_*,R]$ respectively; since moreover
		$\Psi>0$ on $(-\infty,\ell)$ and $\tilde\Psi^{\sigma}(\tilde\ell-\sigma)=0$,
		a standard compactness argument provides $\delta\in(0,\bar\xi/2)$ such that
		both inequalities persist for every $\sigma\in[\bar\xi-\delta,\bar\xi]$ on
		$[X_*,\tilde\ell-\sigma]$ and $[X_*,R]$ respectively. Outside these sets they
		hold by \eqref{sw-unif} and \eqref{sw-tail}. Hence $\bar\xi-\delta\in S$, a
		contradiction. Therefore $\bar\xi=0$, which is exactly \eqref{sw-slide}.
		\hfill$\square$}

		\rev{We now apply this to the monotonicity of $\psi_L$ and $\varphi_L$.}
		Now we define
		\[
		\psi_{L,s} := \psi_L(x + s), \quad \varphi_{L,s} := \varphi_L(x + s), \text{ for } s > 0.
		\]
		Since \(A\) is nonincreasing on \(\mathbb{R}\), we have
		\[
		\begin{cases}
			-c \psi_{L,s} ' - d_1 \psi_{L,s} '' = (A(x + s) - b_1 \psi_{L,s} - c_1 \varphi_{L,s})  \psi_{L,s} \le (A(x) - b_1 \psi_{L,s}  - c_1 \varphi_{L,s})\psi_{L,s}, & x \le L-s,\\
			-c \varphi_{L,s} ' - d_2 \varphi_{L,s} '' = (a_2 - b_2 \psi_{L,s} - c_2 \varphi_{L,s})  \varphi_{L,s}, & x \in \mathbb{R}.
		\end{cases}
		\]
		\rev{Thus $(\tilde\Psi,\tilde\Phi):=(\psi_{L,s},\varphi_{L,s})$ is a lower
		solution with $\tilde\ell=L-s$, and it obeys \eqref{asym-unified} with the
		same exponents as $(\Psi,\Phi):=(\psi_L,\varphi_L)$ because the two pairs are
		translates of one another. Hence \eqref{sw-slide} applies and gives}
		\[
		\psi_{L,s}(x) \le \psi_L(x) \text{ for } x \le L -s, \text{ and }  \varphi_{L,s}(x) \ge \varphi_L(x) \text{ for } x \in \mathbb{R}.
		\]
		For \(L -s < x < L\), we have \(\psi_{L,s}(x) = 0 < \psi_L(x)\). Since $s$ is arbitrary, we deduce that $\psi_L$ is \rev{nonincreasing} on $(-\infty,L]$ and $\varphi_L$ is \rev{nondecreasing} on $\mathbb{R}$.
		Hence, $\psi'_L(x) < 0$ in $(-\infty, L]$ and $\varphi'_L(x) > 0$ for all $x \in \mathbb{R}$, \rev{by the strong maximum principle applied to the linear equations satisfied by $\psi_L'$ and $\varphi_L'$ (recall that $A$ is Lipschitz and nonincreasing, so $A'\le0$ a.e.) and} thanks to the Hopf boundary lemma.

		Consequently, $(\psi_L, \varphi_L)$ is a solution of \eqref{semi-wave}.

		\rev{\textit{Step 5: Uniqueness of the solution.}}
		Let $(\psi_1,\varphi_1)$ and $(\psi_2,\varphi_2)$ be two solutions of \eqref{semi-wave}.
		\rev{Both belong to the class described in Step~4 with $\ell=\tilde\ell=L$,
		and, by Step~3, both satisfy \eqref{asym-unified} with the same exponents
		$\lambda_0,\eta_1$ and the same $m$ (only the constants $K_\psi,K_\varphi$
		may differ). Moreover, since $A$ is nonincreasing, $(\psi_2,\varphi_2)$ is
		in particular a lower solution of \eqref{sw-class}. Applying
		\eqref{sw-slide} with $(\Psi,\Phi)=(\psi_1,\varphi_1)$ and
		$(\tilde\Psi,\tilde\Phi)=(\psi_2,\varphi_2)$ yields
		\[
		\psi_2\le\psi_1 \text{ on }(-\infty,L],\qquad \varphi_2\ge\varphi_1 \text{ on }\mathbb{R}.
		\]
		Exchanging the roles of the two pairs gives the reverse inequalities, so
		$(\psi_1,\varphi_1)=(\psi_2,\varphi_2)$. This completes the uniqueness proof.}

		\textbf{(ii) Monotonicity of $\psi'_L(L)$.}  
		Take $0\le L_1<L_2$ and set
		$\psi_2(x)=\psi_{L_2}(x+L_2-L_1)$, $\varphi_2(x)=\varphi_{L_2}(x+L_2-L_1)$.
		Since $A$ is nonincreasing, $A(x+L_2-L_1)\le A(x)$ for all $x$.  
		A direct substitution shows that $(\psi_2,\varphi_2)$ satisfies
		\[
		\begin{cases}
			-c \psi_2' - d_1 \psi_2'' = (A(x + L_2 - L_1) - b_1\psi_2 - c_1\varphi_2)\psi_2, \quad-\infty < x < L_1,\\
			-c \varphi_2' - d_2 \varphi_2'' = (a_2 - b_2\psi_2 - c_2\varphi_2)\varphi_2, \qquad -\infty < x < +\infty,\\
			\psi_2(-\infty) = \frac{a_1}{b_1}, ~ \psi_2(x)=0 \text{ for } x\ge L_1, ~ \psi_2'(x)<0 \text{ for } x\le L_1,\\	
			\varphi_2(-\infty) = 0, ~ \varphi_2(+\infty)=\frac{a_2}{c_2}, ~ \varphi_2'(x)>0 \text{ for } x\in\mathbb{R}.
		\end{cases}	
		\]
		Due to $A(x + L_2 - L_1) \leq A(x),$ \rev{the pair $(\psi_2,\varphi_2)$ is a
		lower solution of \eqref{sw-class} with $\ell=\tilde\ell=L_1$, and by Step~3
		(applied to $(\psi_{L_2},\varphi_{L_2})$ and then translated) it obeys
		\eqref{asym-unified} with the same exponents as $(\psi_{L_1},\varphi_{L_1})$.
		Hence the sliding lemma \eqref{sw-slide} of Step~4 gives}
		\[
		\psi_{L_1}(x)\ge \psi_2(x)\;\;(x\le L_1),\qquad
		\varphi_{L_1}(x)\le \varphi_2(x)\;\;(x\in\mathbb{R}).
		\]
		Because $\psi_{L_1}(L_1)=\psi_2(L_1)=0$ and $\psi_{L_1}(x)>\psi_2(x)$ for
		$x<L_1$ (strict inequality follows from the strong maximum principle\rev{,
		which applies because $\psi_{L_1}\not\equiv\psi_2$: indeed $A$ is strictly
		decreasing on $[0,l_0]$ and $L_2>L_1$, so $A(x+L_2-L_1)<A(x)$ on a nonempty
		open subset of $(-\infty,L_1)$}), the
		Hopf boundary lemma yields $\psi'_{L_1}(L_1)<\psi'_2(L_1)=\psi'_{L_2}(L_2)$.
		Thus the map $L\mapsto\psi'_L(L)$ is strictly increasing.
		
		\textbf{(iii) Existence and uniqueness of $L_0$.}  
		We first show $\displaystyle\lim_{L\to\infty}\psi'_L(L)=0$.
		\rev{Set $w_L(x):=\psi_L(x+L)$ for $x\le0$, so that $w_L(0)=0$ and
		$0<w_L\le a_1/b_1$. Since $b_1w_L+c_1\varphi_L(\cdot+L)\ge0$, we have
		\[
		-cw_L'-d_1w_L''=\bigl(A(x+L)-b_1w_L-c_1\varphi_L(x+L)\bigr)w_L\le A(x+L)\,w_L .
		\]
		Let $I_L:=(l_0-L,0)$. For $x\in I_L$ we have $x+L>l_0$, hence $A(x+L)=a_0<0$,
		so, with $\kappa:=-a_0>0$,
		\[
		d_1w_L''+cw_L'-\kappa w_L\ \ge\ 0\qquad\text{ in } I_L .
		\]
		Let $\beta:=\dfrac{-c-\sqrt{c^2+4d_1\kappa}}{2d_1}<0$, so that
		$d_1\beta^2+c\beta-\kappa=0$, and put $W(x):=\frac{a_1}{b_1}e^{\beta(x-l_0+L)}$.
		Then $d_1W''+cW'-\kappa W=0$, $W(l_0-L)=\frac{a_1}{b_1}\ge w_L(l_0-L)$ and
		$W(0)>0=w_L(0)$. Since the zeroth-order coefficient $-\kappa$ is negative,
		the maximum principle applied to $w_L-W$ on $I_L$ yields $w_L\le W$ there,
		whence
		\[
		\|w_L\|_{L^\infty([-2,0])}\ \le\ \frac{a_1}{b_1}\,e^{\beta(L-l_0-2)}\ \longrightarrow\ 0
		\qquad (L\to+\infty).
		\]
		Because the coefficients of the equation for $w_L$ are bounded uniformly in
		$L$, interior elliptic estimates up to the boundary give
		$\|w_L\|_{C^1([-1,0])}\to0$, and in particular
		$\psi'_L(L)=w_L'(0)\to0$ as $L\to+\infty$. Consequently}
		\[
		\lim_{L\to\infty}\bigl(-\mu(A(L))\psi'_L(L)\bigr)=0.
		\]

		For $L=0$, the domain for $\psi$ is $(-\infty, 0]$. Since $A(x) \equiv a_1$ for all $x \le 0$, the system \eqref{semi-wave} is identical to the semi-wave system \eqref{semi-wave0}. By the uniqueness established in Theorem \ref{th:semi-wave0}, we necessarily have $(\psi_0, \varphi_0) \equiv (\Psi_c, \Phi_c)$. 
		Consequently, $\psi_0'(0) = \Psi_c'(0)$. Applying Theorem \ref{th:semi-wave0} again, the function $c\mapsto -\mu(a_1)\Psi_c'(0)$
		is strictly decreasing on $[0,c_0]$ and satisfies $-\mu(a_1)\Psi_{c_0}'(0)=c_0$.
		Hence $-\mu(a_1)\Psi_c'(0) \ge c$ for all $c\le c_0$, with equality iff $c=c_0$. Since $A(0) = a_1$, we obtain
		\[
		-\mu(A(0))\psi_0'(0) = -\mu(a_1)\Psi_c'(0) \ge c.
		\]
		
		Now, set $f(L)=-\mu(A(L))\psi'_L(L)$ for $L \ge 0$. From (ii), $L\mapsto-\psi'_L(L)$ is positive and strictly decreasing. Since $A$ is nonincreasing and $\mu$ is increasing, $\mu(A(L))$ is positive and nonincreasing. Hence, $L\mapsto f(L)$ is strictly decreasing on $[0, +\infty)$. 
		
		\rev{The map $L\mapsto f(L)$ is continuous on $[0,+\infty)$: if $L_n\to L$,
		then by the uniform bounds \eqref{sw2} and elliptic estimates a subsequence
		of $(\psi_{L_n},\varphi_{L_n})$ converges in $C^2_{\rm loc}$ to a solution of
		\eqref{semi-wave} with parameter $L$, which by part~(i) must be
		$(\psi_L,\varphi_L)$; since the limit is independent of the subsequence, the
		whole sequence converges and $\psi'_{L_n}(L_n)\to\psi'_L(L)$, while
		$\mu(A(L_n))\to\mu(A(L))$ by the continuity of $\mu$ and $A$.} We have established that $f(0) \ge c$ and $\lim_{L\to\infty} f(L) = 0$. By the Intermediate Value Theorem, there exists a unique $L_0 \ge 0$ such that $f(L_0) = c$. Furthermore, $L_0 = 0$ if and only if $f(0) = c$, which, as shown above, happens if and only if $c = c_0$. This completes the proof of Theorem \ref{th:semi-wave}.
	\end{proof}
	
	We divide the proof of Proposition \ref{prop:speed} into three subsections according to the value of $c$.
	
	\subsection{Case 1: $c < c_0$} In this subsection, we assume $c < c_0$.
	
	\begin{lemma}\label{lem:limsup} Assume $c < c_0$. We have
		\[
		\limsup_{t \to +\infty} \frac{h(t)}{t} \le c.
		\]
	\end{lemma}
	
	\begin{proof}
		The proof is carried out in two steps.
		
		\textit{Step 1.  A pointwise lower bound for \(v\) at large time.}  
		For any small \(\delta>0\) and any \(T_0>0\), we claim that there exist \(T>T_0\) and \(M>0\) such that
		\[
		v(T,x) \ge \frac{a_2-\delta}{c_2}\quad\text{ for all } x\ge M.
		\]
		
		The argument is standard: we choose $\sigma_0$ such that \(0<\sigma_0<\min\bigl\{\liminf_{x\to\infty}v_0(x),\; a_2/c_2\bigr\}\) and consider the auxiliary problem
		\[
		\begin{cases}
			w_t-d_2w_{xx}=(a_2-c_2w)w, & t>0,\;x>0,\\
			w(t,0)=0,\; w(0,x)=\sigma_0.
		\end{cases}
		\]
		Its solution $w$ converges locally uniformly to the positive stationary solution \(w_*(x)\) of 
		\[
		\begin{cases}
			-d_2w_*''=(a_2-c_2w_*)w_*, & x>0,\\
			w_*(0)=0.
		\end{cases}
		\]
		Moreover, $w'_*(x) > 0$ and $w_*(+\infty) = \frac{a_2}{c_2}$. Hence we can find \(M_1>0\) and \(T>T_0\) such that \(w(t,M_1)\ge (a_2-\delta)/c_2\) for all $t \geq T$.
		Applying the maximum principle to the equation satisfied by $w_x$, we deduce $w_x \ge 0$ for $t >0$ and $x>0$. It follows that
		\begin{equation}\label{limsup1}
			w(t,x) \geq w(t,M_1)\ge \frac{a_2 - \delta}{c_2} \text{ for all } t  \geq T \text{ and } x \geq M_1.
		\end{equation}
		Choose \(M_2\) with \(v_0(x)>\sigma_0\) for \(x\ge M_2\) and set \(M_3=\max\{M_2,h(T)\}\).  Then \(h(t)\le M_3\) for \(t\in[0,T]\) and hence $v$ satisfies
		\[
		v_t - d_2v_{xx} = (a_2 - c_2v)v \text{ for } 0 < t \leq T, ~ x > M_3, \quad v(0,x) > \sigma_0 \text{ for } x > M_3.	
		\]
		
		The function \(\tilde w(t,x)=w(t,x-M_3)\) satisfies \(\tilde w_t-d_2\tilde w_{xx}=(a_2-c_2\tilde w)\tilde w\) for \(0<t\le T\), \(x>M_3\), with \(\tilde w(t,M_3)=0\le v(t,M_3)\) and \(\tilde w(0,x)=\sigma_0<v(0,x)\).  The comparison principle gives
		\begin{equation}\label{limsup2}
			v(t, x) \geq \tilde{w}(t, x) = w(t, x - M_3) \quad \text{ for } 0 < t \leq T, ~ x > M_3.
		\end{equation}
		Taking \(t=T\) and \(M:=M_1+M_3\) and using \eqref{limsup1}, \eqref{limsup2} yield the desired bound.
		
		\textit{Step 2. Construction of an upper solution and conclusion.}
		For each $\delta>0$, we consider the problem
		\begin{equation}\label{limsup3}
			\begin{cases}
				-c \psi' - d_1 \psi'' = (A(x)+2\delta - b_1\psi - c_1\varphi)\psi, ~\quad-\infty < x < L_\delta,\\
				-c \varphi' - d_2 \varphi'' = (a_2 - \delta - b_2\psi - c_2\varphi)\varphi, \quad\qquad -\infty < x < +\infty,\\
				\psi(-\infty) = \frac{a_1 + 2\delta}{b_1}, ~ \psi(x)=0 \text{ for } x\ge L_\delta, ~ \psi'(x)<0 \text{ for } x\le L_\delta,\\	
				\varphi(-\infty) = 0, ~ \varphi(+\infty)=\frac{a_2 - \delta}{c_2}, ~ \varphi'(x)>0 \text{ for } x\in \mathbb{R}.
			\end{cases}
		\end{equation}
		This can be interpreted as a perturbed problem of \eqref{semi-wave}\rev{,
		obtained by replacing $A$ by $A+2\delta$ and $a_2$ by $a_2-\delta$. For
		$\delta>0$ small all the structural assumptions are preserved: $A+2\delta$
		still satisfies the hypotheses on $A$ with $a_0+2\delta<0<a_1+2\delta$,
		condition \eqref{superior} continues to hold by continuity, and $\mu$ is
		defined at $A(L_\delta)+2\delta>a_1$ thanks to the extension of $\mu$
		introduced in Section~\ref{sec_1}}. By Theorem \ref{th:semi-wave}, there exists $c_{0,\delta}>0$ such that if $c \le c_{0,\delta}$ we can find $L_\delta \ge 0$ so that \eqref{limsup3} has a solution \((\psi_\delta,\varphi_\delta)\) with
		\begin{equation}\label{limsup4}
			-\mu\bigl(A(L_\delta)+2\delta\bigr)\,\psi_\delta'(L_\delta)=c.
		\end{equation}
		Moreover, $L_\delta=0$ if and only if $c=c_{0,\delta}$.
		
		Due to the assumption $c < c_0$ and $c_{0,\delta} \to c_0$ as $\delta \to 0$ by continuity, we can fix small $\delta$ such that $c < c_{0,\delta}$. Then we choose $L_\delta \ge 0$ and obtain the solution \((\psi_\delta,\varphi_\delta)\) to problem \eqref{limsup3} with property \eqref{limsup4} as mentioned above.
		
		On the other hand, from the bounds for the original solution \((u,v,h)\) (Theorem \ref{th:bounds}) we can choose \(T_0>0\) so large that
		\[
		u(t,x)\le\frac{a_1+\delta}{b_1}\qquad\text{ for all } t\ge T_0,\;0\le x\le h(t).
		\]
		Applying Step 1 with this $T_0$ and the same $\delta$, we obtain \(T>T_0\) and \(M>0\) with \(v(T,x)\ge (a_2-\delta)/c_2\) for all \(x\ge M\).
		Because \(\psi_\delta(L_\delta)=0\) and \(\psi_\delta\) is decreasing, \rev{and
		since $\psi_\delta(-\infty)=\frac{a_1+2\delta}{b_1}>\frac{a_1+\delta}{b_1}$
		and $\varphi_\delta(-\infty)=0<\inf_{[0,M]}v(T,\cdot)$,} we can pick \(R>0\) so large that
		\[
		L_\delta+R>\max\{h(T),l_0\},\qquad
		\psi_\delta(h(T)\rev{{}-cT}-R)>\frac{a_1+\delta}{b_1},\qquad
		\varphi_\delta(M\rev{{}-cT}-R)<\inf_{0\le x\le M}v(T,x).
		\]
		
		\rev{Now we compare on the time interval $[T,+\infty)$, \emph{keeping the
		original time variable}, which avoids any relabelling of the shifting
		environment. We define, for $t\ge T$,}
		\begin{align*}
			\overline{h}(t)&=c t+L_\delta+R ,\\
			\overline{u}(t,x)&=\psi_\delta\bigl(x-c t-R\bigr) \quad\text{ for } 0\le x \le \overline{h}(t),\\
			\underline{v}(t,x)&=\varphi_\delta\bigl(x-c t-R\bigr)\quad\text{ for } x \ge 0.
		\end{align*}
		We claim that \((\overline{u},\underline{v},\overline{h})\) is an upper solution of the original system \eqref{main} \rev{on $[T,+\infty)$}. Indeed,
		\begin{align*}
			\overline{u}_t-d_1\overline{u}_{xx}
			&=-c\psi_\delta'-d_1\psi_\delta''
			=\bigl(A(x-c t-R)+2\delta-b_1\overline{u}-c_1\underline{v}\bigr)\overline{u} \ge\bigl(A(x-ct)-b_1\overline{u}-c_1\underline{v}\bigr)\overline{u},\\
			\underline{v}_t-d_2\underline{v}_{xx}
			&=-c\varphi_\delta'-d_2\varphi_\delta''
			=\bigl(a_2-\delta-b_2\overline{u}-c_2\underline{v}\bigr)\underline{v} \le\bigl(a_2-b_2\overline{u}-c_2\underline{v}\bigr)\underline{v},\\
			\overline{u}_x(t,0)&=\psi_\delta'( -c t-R)\le0,\\
			\underline{v}_x(t,0)&=\varphi_\delta'( -c t-R)\ge0,\\
			\overline{u}(t,x)&=0 \text{ for } x \ge \overline{h}(t),\\
			\overline{h}'(t)&=c=-\mu\bigl(A(L_\delta)+2\delta\bigr)\,\psi_\delta'(L_\delta)\ge -\mu(A(L_\delta+R))\,\psi_\delta'(L_\delta) = -\mu\bigl(A(\overline{h}(t)-ct)\bigr)\,\overline{u}_x(t,\overline{h}(t)),\\
			\rev{\overline{h}(T)}&\rev{{}=cT+L_\delta+R>h(T)},\\
			\rev{\overline{u}(T,x)}&\rev{{}=\psi_\delta(x-cT-R)\ge\psi_\delta(h(T)-cT-R)>\frac{a_1+\delta}{b_1}\ge u(T,x) \text{ for } 0\le x\le h(T)},\\
			\rev{\underline{v}(T,x)}&\rev{{}=\varphi_\delta(x-cT-R) \le\begin{cases}
				\varphi_\delta(M-cT-R)<\inf_{0\le y\le M}v(T,y)\le v(T,x) & \text{ if } 0\le x\le M,\\
				\varphi_\delta(+\infty)=\frac{a_2-\delta}{c_2}\le v(T,x) & \text{ if } x\ge M.
			\end{cases}}
		\end{align*}
		
		Therefore, $(\overline{u},\underline{v},\overline{h})$ is indeed an upper solution \rev{on $[T,+\infty)$. Note that the proof of Theorem \ref{th:comparison} is insensitive to the choice of the initial time.}  By the comparison principle (Theorem \ref{th:comparison}) we obtain
		\[
		\rev{h(t)\le\overline{h}(t)=ct+L_\delta+R\qquad\forall t\ge T.}
		\]
		Consequently,
		\[
		\limsup_{t\to\infty}\frac{h(t)}{t}\le c.
		\]
		
		This completes the proof of Lemma \ref{lem:limsup}.
	\end{proof}
	
	\begin{lemma}\label{lem:liminf} Assume $c < c_0$. We have
		\[
		\liminf_{t \to +\infty} \frac{h(t)}{t} \ge c.
		\]
	\end{lemma}	
	
	To prove Lemma \ref{lem:liminf}, we construct a lower solution $(\underline{u}, \overline{v}, \underline{h})$ for \eqref{main}. \rev{Here $A$ is replaced by $A-\delta$ and $a_2$ by $a_2+\delta$; for $\delta>0$ small all the structural assumptions are again preserved, \eqref{superior} still holds, and $\mu$ is evaluated at $A(L_\delta)-\delta<a_0$ through the extension of $\mu$ introduced in Section~\ref{sec_1}.} By Theorem \ref{th:semi-wave}, for small $\delta>0$ with $\frac{a_1 - \delta}{b_1}>\frac{a_2}{b_2}$, there exists \(L_\delta \ge 0\) such that 
	\[
	-\mu(A(L_\delta)-\delta)\psi_\delta'(L_\delta) = c,
	\]
	where \((\psi_\delta, \varphi_\delta)\) is the unique solution of the following problem	
	\begin{equation}\label{aux}
		\begin{cases}
			-c \psi' - d_1 \psi'' = (A(x) - \delta - b_1\psi - c_1\varphi)\psi, \quad-\infty < x < L_\delta,\\
			-c \varphi' - d_2 \varphi'' = (a_2 + \delta - b_2\psi - c_2\varphi)\varphi, \qquad -\infty < x < +\infty,\\
			\psi(-\infty) = \frac{a_1 - \delta}{b_1}, ~ \psi(x)=0 \text{ for } x\ge L_\delta, ~ \psi'(x)<0 \text{ for } x\le L_\delta,\\	
			\varphi(-\infty) = 0, ~ \varphi(+\infty)=\frac{a_2 + \delta}{c_2}, ~ \varphi'(x)>0 \text{ for } x\in \mathbb{R}.
		\end{cases}
	\end{equation}
	
	\begin{lemma}\label{lm:aux}
		For any \(x_0 < L_\delta\), there exist a constant \(x_1 = x_1(x_0) < x_0\) and a function \(\tilde{\varphi}_{\delta} \in C^1(\mathbb{R})\) such that
		\begin{align*}
			\limsup_{x_0 \to -\infty}[x_0-x_1(x_0) ] < +\infty, ~ \tilde{\varphi}_\delta' \ge 0, ~ \tilde{\varphi}_\delta \ge \varphi_\delta, \qquad -\infty < x < +\infty,\\
			\tilde{\varphi}_\delta(x) = \varphi_\delta(x) \text{ for } x\ge x_0,\quad \tilde{\varphi}_\delta(x) = \tilde{\varphi}_\delta(x_1) > 0 \text{ for } x \le x_1,\\
			-c \psi_\delta' -d_1 \psi_\delta'' \le (A(x) - b_1 \psi_\delta - c_1 \tilde{\varphi}_\delta) \psi_\delta, \quad-\infty < x < L_\delta,\\
			-c\tilde{\varphi}_\delta' - d_2 \tilde{\varphi}_\delta'' \ge (a_2 - b_2\psi_\delta - c_2\tilde{\varphi}_\delta) \tilde{\varphi}_\delta, \qquad -\infty < x < +\infty.
		\end{align*}
	\end{lemma}
	
	\begin{proof}
		We follow an idea from the proof of \cite[Lemma 3.3]{MR3609207}.
		As in Step 4 of the proof of Theorem \ref{th:semi-wave} (i), we have
		\begin{equation}\label{laux1}
			\varphi_{\delta}(x) = C_0 e^{\lambda_0 x}(1+o(1)),\qquad
			\varphi_{\delta}'(x) = C_0\lambda_0 e^{\lambda_0 x}(1+o(1))\quad\text{as }x\to-\infty,
		\end{equation}
		for some $C_0,\lambda_0>0$.		
		Fix $\lambda<0$, with its value to be determined later, and define, for $x\le x_0$,
		\[
		\hat{\varphi}(x) := \varphi_{\delta}(x) + \varphi_{\delta}(x_0)\bigl[e^{\lambda(x-x_0)} - 1 - \lambda(x-x_0)\bigr].
		\]
		By \eqref{laux1} we have, for $x_0$ sufficiently negative and $x\le x_0$,
		\[
		\begin{aligned}
			\hat{\varphi}'(x) &= \varphi_{\delta}'(x) + \varphi_{\delta}(x_0)\bigl[\lambda e^{\lambda(x-x_0)} - \lambda\bigr] \\
			&= C_0\lambda_0 e^{\lambda_0 x}(1+o(1)) + C_0 e^{\lambda_0 x_0}(1+o(1))\bigl[\lambda e^{\lambda z} - \lambda\bigr] \\
			&= C_0 e^{\lambda_0 x_0}\bigl[\lambda_0 e^{\lambda_0 z} + \lambda(e^{\lambda z}-1)\bigr](1+o(1)),
		\end{aligned}
		\]
		where $z:=x-x_0$. Since $\lambda<0<\lambda_0$, the strictly increasing function
		$f(z):=\lambda_0 e^{\lambda_0 z} + \lambda(e^{\lambda z}-1)$ has a unique negative zero $z_0$.
		It follows that, for all $x_0$ sufficiently negative,
		\[
		\hat{\varphi}'(x_0+z_0-1) < 0.
		\]
		Hence there exists $x_1\in(x_0+z_0-1,x_0)$ such that
		\[
		\hat{\varphi}'(x)>0=\hat{\varphi}'(x_1) \quad \text{ for all } x_1 < x \le x_0.
		\]
		
		Define
		\[
		\tilde{\varphi}_{\delta}(x) := 
		\begin{cases}
			\hat{\varphi}(x_1), & x\le x_1,\\
			\hat{\varphi}(x), & x_1\le x\le x_0,\\
			\varphi_{\delta}(x), & x\ge x_0.
		\end{cases}
		\]
		Then clearly $\tilde{\varphi}_{\delta}\in C^1(\mathbb{R})$, $\tilde{\varphi}_{\delta}(x_1)>\varphi_{\delta}(x_1)>0$, and
		\[
		\lim_{x_0\to-\infty}[x_0-x_1(x_0)]\le |z_0-1| < +\infty,\quad
		\tilde{\varphi}_{\delta}'\ge 0,\quad \tilde{\varphi}_{\delta}\ge\varphi_{\delta} \text{ for all } x\in\mathbb{R}.
		\]
		Moreover, it is also easily seen that
		\[
		\lim_{x_0\to-\infty}\|\tilde{\varphi}_\delta - \varphi_{\delta}\|_{L^\infty(\mathbb{R})}=0.
		\]
		Thus for $x_0$ sufficiently negative,
		\[
		-c \psi_{\delta}' - d_1 \psi_{\delta}'' = (A(x) - \delta - b_1\psi_{\delta} - c_1\varphi_{\delta})\psi_{\delta} \le (A(x) - b_1 \psi_\delta - c_1 \tilde{\varphi}_\delta) \psi_\delta, \quad-\infty < x < L_\delta,
		\]
		
		To complete the proof, it remains to show that, for every fixed $x_0$ sufficiently negative, in the weak sense,
		\begin{equation}\label{laux2}
			-c\tilde{\varphi}_\delta' - d_2 \tilde{\varphi}_\delta'' \ge (a_2 - b_2\psi_\delta - c_2\tilde{\varphi}_\delta) \tilde{\varphi}_\delta, \qquad -\infty < x < +\infty.
		\end{equation}
		Since $\tilde{\varphi}_{\delta}$ is $C^1$, it suffices to show the above inequality for $x<x_1$, $x\in(x_1,x_0)$ and $x>x_0$ separately.
		
		For $x>x_0$, \eqref{laux2} follows directly from \eqref{aux}. Since $\psi_\delta(+\infty)=\frac{a_1 - \delta}{b_1}>\frac{a_2}{b_2}$, we have
		$a_2 - b_2\psi_\delta - c_2\tilde{\varphi}_\delta<0$ for $x<x_1$ provided that $x_0$ is sufficiently negative. Hence for every fixed $x_0$ sufficiently negative, \eqref{laux2} holds for $x<x_1$.
		
		We next consider the case $x\in(x_1,x_0)$. Denote
		\[
		\eta(x) := e^{\lambda(x-x_0)} - 1 - \lambda(x-x_0).
		\]
		Then $-c\eta' - d_2\eta'' = -d_2\lambda^2 e^{\lambda(x-x_0)} - \lambda c \left(e^{\lambda(x-x_0)} - 1\right) > -d_2\lambda^2 e^{\lambda(x-x_0)}$ for $x<x_0$, and hence, for $x\in(x_1,x_0)$,
		\[
		\begin{aligned}
			-c\tilde{\varphi}_\delta' - d_2 \tilde{\varphi}_\delta'' &= -c {\varphi}_\delta' - d_2 {\varphi}_\delta'' + \varphi_{\delta}(x_0)(-c\eta' - d_2\eta'') \\
			&= (a_2 + \delta - b_2\psi_\delta - c_2\varphi_\delta)\varphi_\delta + \varphi_{\delta}(x_0)(-c\eta' - d_2\eta'') \\
			&> (a_2 - b_2\psi_\delta - c_2\varphi_\delta)\varphi_\delta + \delta\varphi_{\delta} - \varphi_{\delta}(x_0)d_2 \lambda^2 e^{\lambda(x-x_0)}.
		\end{aligned}
		\]
		On the other hand, for such $x$,
		\begin{align*}
			(a_2 - b_2\psi_\delta - c_2\tilde{\varphi}_\delta) \tilde{\varphi}_\delta &= (a_2 - b_2\psi_\delta - c_2\tilde{\varphi}_\delta)\varphi_\delta + \varphi_{\delta}(x_0)\eta (a_2 - b_2\psi_\delta - c_2\tilde{\varphi}_\delta)\\
			&\le (a_2 - b_2\psi_\delta - c_2 {\varphi}_\delta)\varphi_\delta + \varphi_{\delta}(x_0)\eta (a_2 - b_2\psi_\delta)\\
			&= (a_2 - b_2\psi_\delta - c_2 {\varphi}_\delta)\varphi_\delta - \sigma \varphi_{\delta}(x_0)\eta + \varepsilon(x)\varphi_{\delta},
		\end{align*}
		where
		\[
		\sigma := b_2\cdot\frac{a_1 - \delta}{b_1} - a_2 > 0,\qquad
		\varepsilon(x) := \frac{\varphi_{\delta}(x_0)}{\varphi_{\delta}(x)}\eta(x)\left(a_2 - b_2\psi_\delta + b_2\frac{a_1 - \delta}{b_1}\right) \to 0
		\]
		uniformly for $x\in[x_1,x_0]$ as $x_0\to-\infty$.
		
		Therefore \eqref{laux2} will hold for $x\in(x_1,x_0)$ with sufficiently negative $x_0$, provided we can show that
		\begin{equation}\label{laux3}
			\frac{\delta}{2}\varphi_{\delta}(x) - \varphi_{\delta}(x_0) d_2 \lambda^2 e^{\lambda(x-x_0)}
			\ge -\sigma\varphi_{\delta}(x_0)\eta(x) \quad \text{ for all } x_1 \le x \le x_0.
		\end{equation}
		
		By \eqref{laux1}, for $x\in[x_1,x_0]$ and sufficiently negative $x_0$,
		\[
		\varphi_{\delta}(x) = C_0 e^{\lambda_0x}(1+o(1)),\qquad
		\varphi_{\delta}(x_0) = C_0 e^{\lambda_0x_0}(1+o(1)).
		\]
		Thus \eqref{laux3} will hold for all sufficiently negative $x_0$ if for some $\delta_0\in(0,\delta/2)$ and $\sigma_0\in(0,\sigma)$,
		\[
		\delta_0 e^{\lambda_0(x-x_0)} - d_2 \lambda^2 e^{\lambda(x-x_0)}
		+ \sigma_0\bigl[e^{\lambda(x-x_0)} - 1 - \lambda(x-x_0)\bigr] \ge 0 \quad \text{ for all } x_1 \le x \le x_0.
		\]
		It suffices to show that if $\lambda<0$ is chosen such that $|\lambda|$ is sufficiently small then  $f_{\lambda}(t)>0$ for all $t\le0$, where
		\[
		f_{\lambda}(t):=\frac{\delta_0}{d_2} e^{\lambda_0 t} - \lambda^2 e^{\lambda t} + \frac{\sigma_0}{d_2} (e^{\lambda t} - 1 - \lambda t).
		\]
		By setting $\gamma_0 = -\lambda_0 < 0$, $\gamma = -\lambda > 0$, and $z=-t$, the claim is equivalent to: For $\gamma>0$ small, $g_{\gamma}(z)>0$ for all $z\ge0$, where
		\[
		g_{\gamma}(z):=\frac{\delta_0}{d_2} e^{\gamma_0 z} - \gamma^2 e^{\gamma z} + \frac{\sigma_0}{d_2} (e^{\gamma z} - 1 - \gamma z).
		\]
		This fact is elementary and its proof can be found in the proof of \cite[Lemma 3.3]{MR3609207} (see page 284 there).		
		The proof is complete.
	\end{proof}
	
	\begin{proof}[Proof of Lemma \ref{lem:liminf}.]
		For small \(\delta > 0,\) let \((\psi_\delta, \tilde{\varphi}_\delta)\) be given by Lemma \ref{lm:aux}. Let \(V(t)\) be the unique solution of the problem 
		\[
		V' = (a_2 - c_2V)V, \quad V(0) = \|v_0\|_\infty.
		\]
		A simple comparison consideration yields 
		\[
		v(t,x) \le V(t)  \text{ for all } t > 0\text{ and } x \ge 0.
		\]
		Since \(\lim_{t \to +\infty}V(t) = a_2/c_2,\) there exists \(T_0 > 0\) such that
		\[
		v(t,x) < \frac{a_2 + \delta/2}{c_2} \text{ for all } t \ge T_0\text{ and } x \ge 0. 
		\]
		
		We fix $M>L_\delta$ such that
		\[
		\tilde{\varphi}_\delta(M)\ge\frac{a_2+\delta/2}{c_2}.
		\]
		Since $\lim_{t\to\infty} h(t) = +\infty$ and $\lim_{t\to\infty} (u(t,x), v(t,x)) = \left(\frac{a_1}{b_1},0\right)$
		uniformly in any compact subset of $[0, \infty)$ (see \rev{Proposition \ref{prop:spreading}}), we can choose $T>T_0$ such that
		\[
		h(T)>L_\delta,
		\quad \inf_{0\le x\le M}u(t,x) > \frac{a_1- \delta}{b_1},
		\quad \sup_{0\le x\le M}v(t,x)<\tilde{\varphi}_\delta(x_1) \quad\text{ for all } t\ge T.
		\]
		
		\rev{Now we compare on $[T,+\infty)$, keeping the original time variable. We define, for $t\ge T$,}
		\begin{align*}
			\underline{h}(t)&=c \rev{(t-T)}+L_\delta ,\\
			\underline{u}(t,x)&=\psi_\delta\bigl(x-c \rev{(t-T)}\bigr) \quad\text{ for } 0\le x \le \underline{h}(t),\\
			\overline{v}(t,x)&=\tilde{\varphi}_\delta\bigl(x-c \rev{(t-T)}\bigr)\quad\text{ for } x \ge 0.
		\end{align*}
		We claim that \((\underline{u},\overline{v},\underline{h})\) is a lower solution of the original system \eqref{main} \rev{on $[T,+\infty)$}. Indeed, \rev{writing $\zeta:=x-c(t-T)\ge x-ct$ and using that $A$ is nonincreasing, so that $A(\zeta)\le A(x-ct)$,}
		\begin{align*}
			\underline{u}_t-d_1\underline{u}_{xx}
			&=-c\psi_\delta'-d_1\psi_\delta''
			\rev{{}=\bigl(A(\zeta)-\delta-b_1\underline{u}-c_1\varphi_\delta(\zeta)\bigr)\underline{u}}
			\le\bigl(A(x-ct)-b_1\underline{u}-c_1\overline{v}\bigr)\underline{u},\\
			\overline{v}_t-d_2\overline{v}_{xx}
			&=-c\tilde{\varphi}_\delta'-d_2\tilde{\varphi}_\delta''
			\ge\bigl(a_2-b_2\underline{u}-c_2\overline{v}\bigr)\overline{v},\\
			\underline{u}(t,0)&=\psi_\delta\bigl(\rev{-c(t-T)}\bigr) < \frac{a_1 - \delta}{b_1} < u(\rev{t},0),\\
			\overline{v}(t,0)&=\tilde{\varphi}_\delta\bigl(\rev{-c(t-T)}\bigr)\ge\tilde{\varphi}_\delta(x_1) > v(\rev{t},0),\\
			\underline{u}(t,x)&=0 \text{ for } x \ge \underline{h}(t),\\
			\underline{h}'(t)&=c=-\mu\bigl(A(L_\delta)-\delta\bigr)\,\psi_\delta'(L_\delta) \le -\mu\bigl(\rev{A(L_\delta-cT)}\bigr)\,\psi_\delta'(L_\delta) = -\mu\bigl(A(\underline{h}(t)-ct)\bigr)\,\underline{u}_x(t,\underline{h}(t)),\\
			\rev{\underline{h}(T)}&\rev{{}=L_\delta<h(T)},\\
			\rev{\underline{u}(T,x)}&\rev{{}=\psi_\delta(x) < \frac{a_1 - \delta}{b_1} \le u(T,x) \text{ for } 0\le x\le \underline{h}(T)},\\
			\rev{\overline{v}(T,x)}&\rev{{}=\tilde{\varphi}_\delta(x) \ge \begin{cases}
				\tilde{\varphi}_\delta(x_1)>\sup_{0\le y\le M}v(T,y)\ge v(T,x) & \text{ if } 0\le x\le M,\\
				\tilde{\varphi}_\delta(M)\ge\frac{a_2+\delta/2}{c_2}\ge v(T,x) & \text{ if } x> M.
			\end{cases}}
		\end{align*}
		\rev{Here we used $A(L_\delta-cT)\ge A(L_\delta)\ge A(L_\delta)-\delta$ and
		$\psi_\delta'(L_\delta)<0$ in the fourth line, and the fact that the
		\emph{second} alternative at $x=0$ in Theorem~\ref{th:comparison} is the one
		that applies here, since $\underline u_x(t,0)=\psi_\delta'(-c(t-T))<0$.}
		
		Applying the comparison principle (Theorem \ref{th:comparison}), we deduce that
		\[
		\rev{\underline{h}(t) \le h(t) \quad\text{ for all } t \ge T.}
		\]
		Consequently,
		\[
		\liminf_{t \to +\infty}{\frac{h(t)}{t}} \ge c.
		\]		
	\end{proof}	
	
	\begin{proof}[Proof of Proposition \ref{prop:speed} for $c < c_0$]
		The conclusion of Proposition \ref{prop:speed} for $c < c_0$ follows directly from Lemmas \ref{lem:limsup} and \ref{lem:liminf}.
	\end{proof}	
	
	\subsection{Case 2: $c = c_0$}
	
	\begin{proof}[Proof of Proposition \ref{prop:speed} for $c = c_0$]	
		Since $(u,v,h)$ is a solution to \eqref{main}, it is a lower-solution to \eqref{main2}.
		Let $(\tilde{u},\tilde{v},\tilde{h})$ be the unique solution of \eqref{main2}, then the comparison principle yields that
		\[
		u(t,x) \le \tilde{u}(t,x), ~ v(t,x) \ge \tilde{v}(t,x) \quad \text{ for } (t,x) \in [0,+\infty) \times [0,h(t)]
		\]
		and
		\[
		h(t) \le \tilde{h}(t) \quad \text{ for } t\ge0.
		\]
		
		Since $h_\infty=+\infty$, we have $\lim_{t\to+\infty} \tilde{h}(t) = +\infty$\rev{, i.e. spreading occurs for \eqref{main2}}. By \cite[Theorem 1.1]{MR3609207}, we derive
		\begin{equation}\label{upperbound}
			\limsup_{t\to+\infty} \frac{h(t)}{t} \le \lim_{t\to+\infty} \frac{\tilde{h}(t)}{t} = c_0.
		\end{equation}
		
		On the other hand, set $c_\delta := c_0 - \delta$ for a sufficiently small $\delta>0$, and let $(u_\delta,v_\delta,h_\delta)$ be the unique solution of		
		\begin{equation}\label{main3}
			\begin{cases}
				u_t = d_1u_{xx} + (A(x - c_\delta t) - b_1u - c_1 v) u, & t > 0,~ 0 < x < h(t), \\
				v_t = d_2v_{xx} + (a_2 - b_2u - c_2v)v,          & t > 0,~ 0 < x < +\infty, \\
				u_x(t, 0) = v_x(t, 0) = 0, ~ u(t, x) = 0,        & t > 0, ~ h(t) \leq x < +\infty, \\
				h'(t) = -\mu (A(h(t) - c_\delta t))u_x(t, h(t)),          & t > 0, \\
				h(0)=h_0,~ u(0, x) = u_0(x),                       & 0 \leq x \leq h_0, \\
				v(0, x) = v_0(x),                                  & 0 \le x < +\infty.
			\end{cases}
		\end{equation}
		Since $A(x-ct) \ge A(x-c_\delta t)$ and $\mu(A(h(t)-ct)) \ge \mu(A(h(t)-c_\delta t))$, we have that $(u,v,h)$ is an upper-solution to \eqref{main3}. By the comparison principle, we deduce that
		\[
		u(t,x) \ge u_\delta(t,x), ~ v(t,x) \le v_\delta(t,x) \quad \text{ for } (t,x) \in [0,+\infty) \times [0,h_\delta(t)]
		\]
		and
		\[
		h(t) \ge h_\delta(t) \quad\text{ for } t\ge0.
		\]
		
		For $M=\max\left\{\frac{\pi}{2}\sqrt{\frac{d_1}{a_1-a_2c_1/c_2}}, h_0\right\}$, we take $T>0$ such that $h(T) = M+1$. By the continuity of the solution on parameters, we can choose $\delta$ small such that $h_\delta(T) > M$. \rev{This implies that spreading occurs for \eqref{main3}: if it did not, Proposition~\ref{prop:vanishing}, applied to \eqref{main3}, would give $\lim_{t\to\infty}h_\delta(t)\le\frac{\pi}{2}\sqrt{\frac{d_1}{a_1-a_2c_1/c_2}}\le M<h_\delta(T)$, which is absurd because $h_\delta$ is increasing.} Applying Lemma \ref{lem:liminf} to problem \eqref{main3} yields
		\[
		\liminf_{t \to +\infty} \frac{h(t)}{t} \ge \liminf_{t \to +\infty} \frac{h_\delta(t)}{t} \ge c_\delta.
		\]
		Letting $\delta\to 0$, we derive
		\begin{equation}\label{lowerbound}
			\liminf_{t\to+\infty} \frac{h(t)}{t} \ge c_0.
		\end{equation}
		Combining \eqref{upperbound} and \eqref{lowerbound}, we get the desired result.
	\end{proof}
	
	\subsection{Case 3: $c > c_0$}
	
	\begin{proof}[Proof of Proposition \ref{prop:speed} for $c > c_0$]
		Using the same arguments as in Case 2, we have
		\[
		\limsup_{t\to+\infty} \frac{h(t)}{t} \le c_0.
		\]
		
		Thanks to $c_0<c$, there exists $T_0>0$ such that
		\[
		h(t) - ct < 0 \quad \text{ for } t> T_0.
		\]
		Hence for $t>T_0$,
		\[
		A(x-ct) = a_1 \text{ for } x \in [0,h(t)]
		\quad\text{ and }\quad
		\mu(A(h(t)-ct)) = \mu(a_1).
		\]
		
		This implies that $(t,x) \mapsto (u(t+T_0, x), v(t+T_0, x), h(t+T_0))$ solves \eqref{main2} for $t > 0$, with the initial data $u_0(\cdot)$, $v_0(\cdot)$ and $h_0$ replaced by $u(T_0,\cdot)$, $v(T_0,\cdot)$ and $h(T_0)$, respectively. Using \rev{\cite[Theorem 1.1]{MR3609207}, spreading being in force for that problem since $h_\infty=+\infty$}, we obtain
		\[
		\lim_{t\to+\infty} \frac{h(t)}{t} = c_0.
		\]
		This finishes the proof.
	\end{proof}
	
	\subsection{Proof of Theorem \ref{th:superior}}
	Theorem \ref{th:superior} follows directly from Propositions \ref{prop:vanishing}, \ref{prop:spreading}, and \ref{prop:speed}.
	
	\section{Numerical simulations}\label{sec_5}
	
	In this section, we perform simulations to illustrate our theoretical results obtained in Theorems \ref{th:inferior} and \ref{th:superior}. The finite difference method combined with the Newton-Raphson iteration is employed in the numerical experiments.
	
	In all figures below, the parameter values are set to $d_1 = 0.02$, $d_2 = 0.012$, $a_0 = -0.2$, and $l_0 = 1$. The initial population densities are given by $v_0(x) = 0.2 + 0.1 \cos\left(\frac{\pi x}{2}\right)$ and $u_0(x) = 0.3 \cos\left(\frac{\pi x}{2h_0}\right)$, where $h_0$ denotes the initial boundary size.
	
	Let us consider model \eqref{main} with $\mu(x) = 0.1 + 0.2 x$, and 
	\[
	A(\xi) = \begin{cases}
		a_1, & \text{if } \xi \le 0, \\
		a_1 + (a_0 - a_1)\left[ 3\left(\frac{\xi}{l_0} \right)^2 - 2\left(\frac{\xi}{l_0} \right)^3 \right], & \text{if } 0 < \xi < l_0, \\
		a_0, & \text{if } \xi \ge l_0.
	\end{cases}
	\]
	
	We first examine the case where $u$ is an inferior competitor corresponding to the condition $\frac{a_1}{a_2} < \min \{b_1/b_2, c_1/ c_2\}$. The selected parameter values are $a_1 =0.1, b_1 = 0.2, c_1 = 0.18, a_2 = 0.12, b_2 = 0.1, c_2 = 0.14$, with an initial boundary $h_0 = 1.5$ and an environmental shifting speed $c = 0.08$. Under these settings, the simulation indicates that $u$ vanishes in the long run (Figure \ref{fig:u_inf_vanishing}) while the native species $v$ establishes the long-term steady-state $a_2/c_2 \approx 0.86$ (Figure \ref{fig:u_inf_v}).
	
	\begin{figure}[htbp]
		\centering
		\begin{subfigure}{0.45\textwidth}
			\centering
			\includegraphics[width=\linewidth]{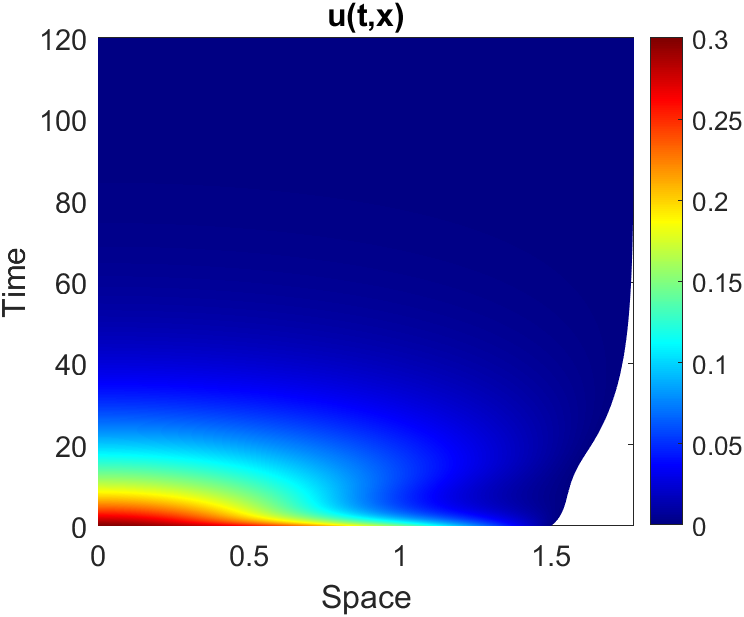}
			\caption{Heatmap of $u(t,x)$}
		\end{subfigure}
		\hfill 
		\begin{subfigure}{0.45\textwidth}
			\centering
			\includegraphics[width=\linewidth]{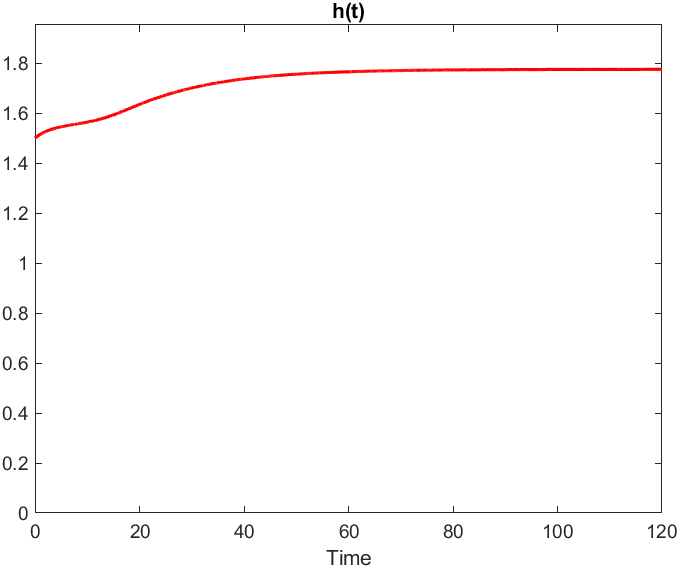}
			\caption{Free boundary $h(t)$}
		\end{subfigure}
		\caption{Vanishing of the inferior competitor species $u$.}
		\label{fig:u_inf_vanishing}
	\end{figure}

	\begin{figure}[htbp]
		\centering
		\includegraphics[scale = 0.55]{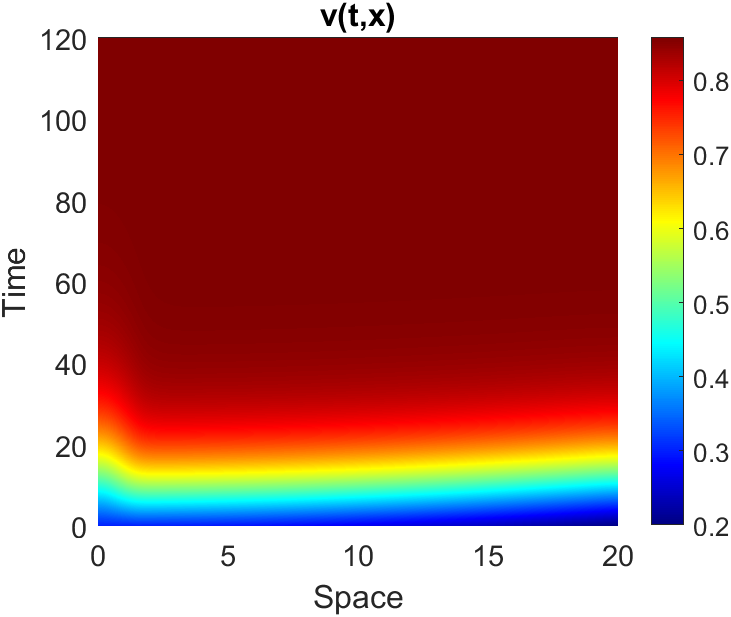}
		\caption{Heatmap of $v(t,x)$ when the inferior competitor $u$ vanishes.}
		\label{fig:u_inf_v}
	\end{figure}
	
	Next, we examine the case where $u$ acts as a superior competitor. Specifically, the parameter values are chosen as $a_1 = 0.3, b_1 = 0.14, c_1 = 0.26, a_2 = 0.1, b_2 = 0.1, c_2 = 0.1$, ensuring that condition $\frac{a_1}{a_2} > \max \{b_1/b_2, c_1/ c_2\}$ is satisfied. Additionally, the environmental shifting speed is fixed at $c = 0.01$ for the two scenarios below.
	
	First, $h_0$ is taken as $0.2$. The numerical simulations indicate that $u$ vanishes in the long run and $h_\infty \approx 0.78 \le \frac{\pi}{2} \sqrt{\frac{d_1}{a_1 - a_2c_1/c_2}}$ \rev{$\approx1.11$} (see Figure \ref{fig:u_sup_vanishing}). Furthermore, the native species $v$ quickly approaches the long-term steady state $\frac{a_2}{c_2} = 1.$
	
	\begin{figure}[htbp]
		\centering
		\begin{subfigure}{0.45\textwidth}
			\centering
			\includegraphics[width=\linewidth]{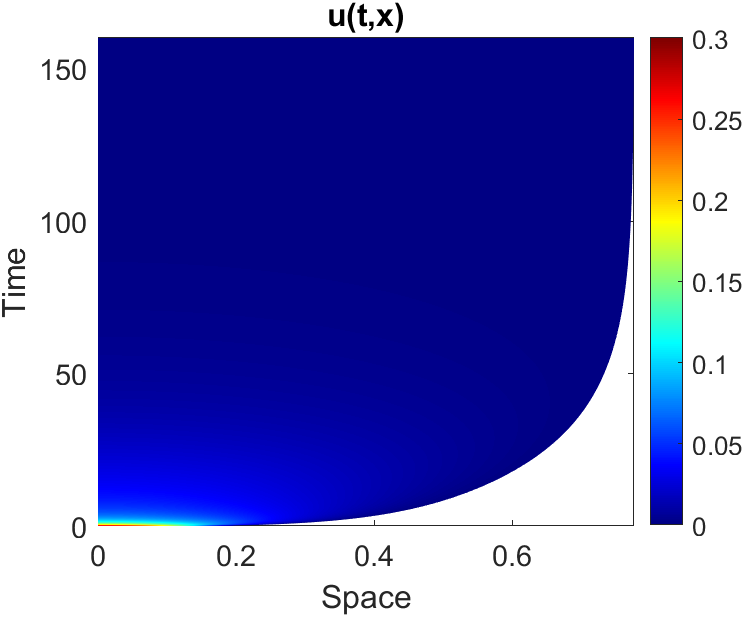}
			\caption{Heatmap of $u(t,x)$}
		\end{subfigure}
		\hfill 
		\begin{subfigure}{0.45\textwidth}
			\centering
			\includegraphics[width=\linewidth]{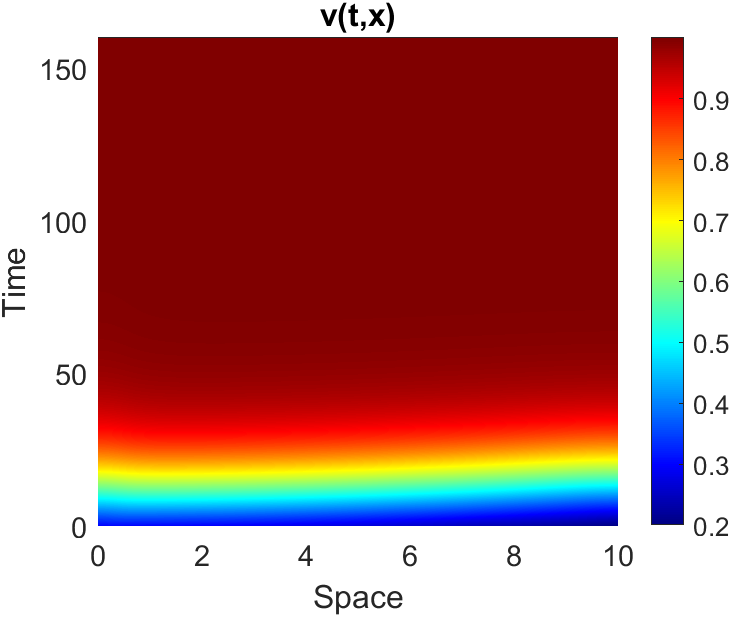}
			\caption{Heatmap of $v(t,x)$}
		\end{subfigure}
		\caption{Vanishing of the superior competitor  $u$ with $h_0$ = 0.2.}
		\label{fig:u_sup_vanishing}
	\end{figure}
	
	Second, we choose $h_0 = 1$ and obtain the case in which the invasive species $u$ establishes the steady state $\frac{a_1}{b_1} \approx 2.14$ while the native species $v$ goes to extinction in the long run. See Figure \ref{fig:u_sup_spreading} for details. \rev{Note that here $h_0=1$ is still smaller than the threshold $\frac{\pi}{2}\sqrt{\frac{d_1}{a_1-a_2c_1/c_2}}\approx1.11$ appearing in Theorem \ref{th:superior}(i); the simulation therefore also illustrates that this necessary condition for vanishing is not sufficient.}
	
	\begin{figure}[htbp]
		\centering
		\begin{subfigure}{0.45\textwidth}
			\centering
			\includegraphics[width=\linewidth]{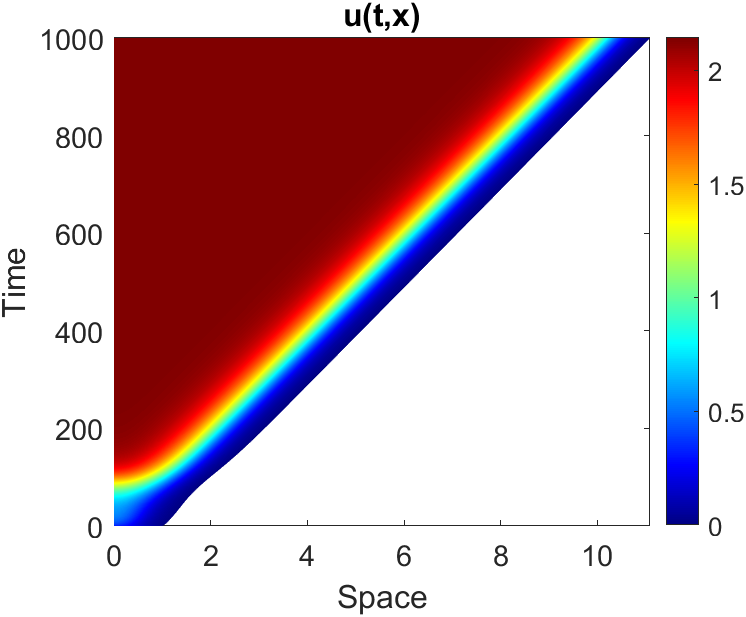}
			\caption{Heatmap of $u(t,x)$}
		\end{subfigure}
		\hfill 
		\begin{subfigure}{0.45\textwidth}
			\centering
			\includegraphics[width=\linewidth]{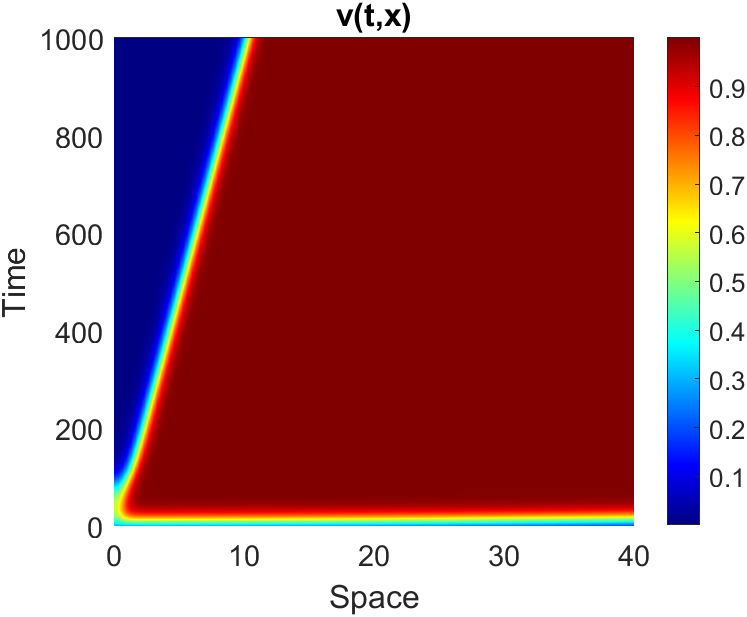}
			\caption{Heatmap of $v(t,x)$}
		\end{subfigure}
		\caption{Spreading of the superior competitor $u$ and extinction of the native species $v$ with $h_0 = 1$.}
		\label{fig:u_sup_spreading}
	\end{figure}
	
	When spreading occurs, we find that $h_\infty = +\infty$ and the asymptotic spreading speed $\frac{h(t)}{t}$ approaches $c = 0.01$. \rev{The convergence in this regime is rather slow: the computed ratio still equals $0.0111$ at $t\approx1000$ and is decreasing, in agreement with the fact that, when $c<c_0$, the free boundary satisfies $h(t)=ct+O(1)$, so that $h(t)/t-c$ decays only like $1/t$.}
	Interestingly, if the shifting speed is increased (for instance, $c = 0.06$), numerical simulations indicate that $\frac{h(t)}{t} \to \rev{0.041}$ for sufficiently large $t$. The asymptotic ratio $\frac{h(t)}{t}$ evaluated at selected values of $c$ is shown in Table \ref{tab:speed} and Figure \ref{fig:hovert_diff_c}. \rev{The values stabilise at $\min\{c,c_0\}$ with $c_0\approx0.041$: they follow $c$ closely for $c\le0.04$ and saturate at $\approx0.041$ for $c\ge0.05$, independently of $c$. This confirms $\frac{h(t)}{t} \to \min\{c, c_0\}$ as $t\to+\infty$, as predicted by Theorem \ref{th:superior}(ii).}
	
	\begin{table}[htbp]
		\centering
		\begin{tabular}{c @{\hskip 1.5cm} c}
			$c$ & $h(t)/t$  when  $t \approx 1000$ \\
			\hline\hline
			0.01 & 0.0111\\
			0.02 & 0.0208\\
			0.03 & 0.0306\\
			0.04 & 0.0400\\
			0.05 & 0.0408 \\
			0.06 & 0.0410 \\
			0.07 & 0.0411\\
			0.08 & 0.0412
		\end{tabular}
		\caption{The asymptotic spreading speed $\frac{h(t)}{t}$ with selected values of $c$\rev{. The values are consistent with $\min\{c,c_0\}$ and $c_0\approx0.041$; the small excess for $c\le0.04$ is due to the slow, $O(1/t)$, convergence in that regime}.}
		\label{tab:speed}
	\end{table}
	
	\begin{figure}[htbp]
		\centering	
		\includegraphics[scale = 0.8]{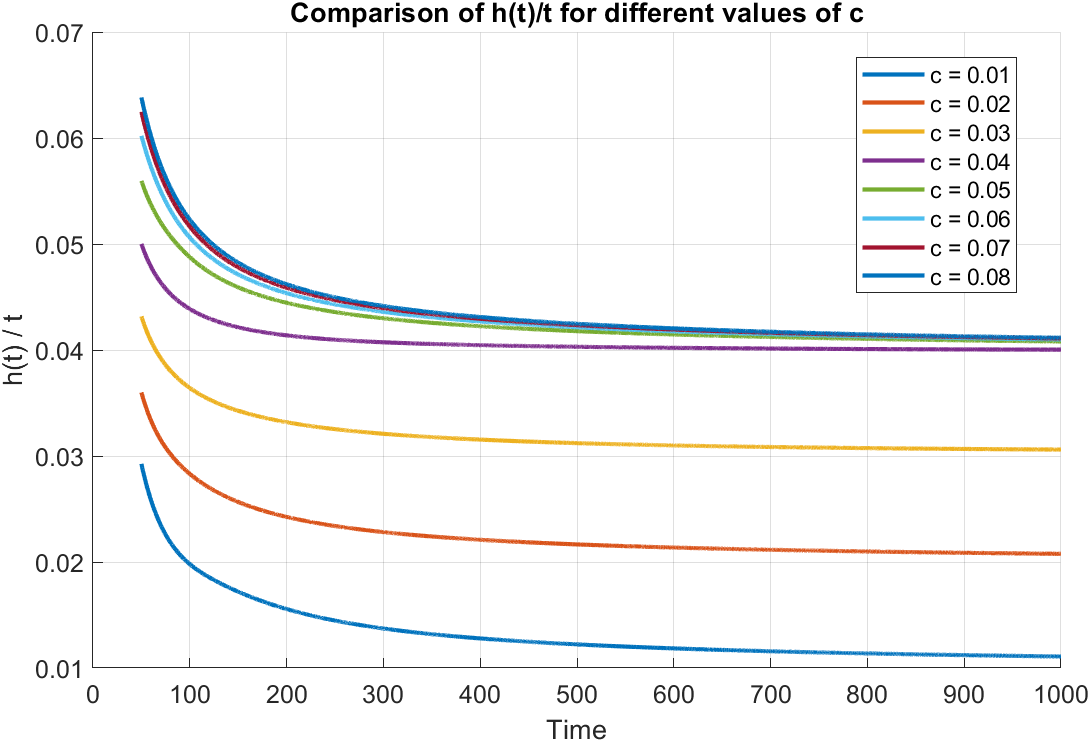}
		\caption{Comparison of $\frac{h(t)}{t}$ for $t \in [50, 1000]$ with different values of $c$.}
		\label{fig:hovert_diff_c}
	\end{figure}
	
	\bibliographystyle{abbrv}
	\bibliography{ref}
	
\end{document}